\documentclass[10pt]{amsart}
\usepackage{amsmath,amssymb,latexsym,esint,cite,mathrsfs}
\usepackage{verbatim,wasysym,cite}
\usepackage[left=2.7cm,right=2.7cm,top=3cm,bottom=3cm]{geometry}
\usepackage[latin1]{inputenc}
\usepackage{microtype}
\usepackage{color,enumitem,graphicx}
\usepackage[colorlinks=true,urlcolor=blue, citecolor=red,linkcolor=blue,
linktocpage,pdfpagelabels, bookmarksnumbered,bookmarksopen]{hyperref}
\usepackage[hyperpageref]{backref}
\usepackage[english]{babel}

\usepackage{textcomp}
\usepackage{mathrsfs}
\usepackage{amsthm}
\usepackage{esint}

\makeatletter

\numberwithin{equation}{section}
\newtheorem{thm}{Theorem}[section]
\newtheorem{lemma}[thm]{Lemma}
\newtheorem{rem}[thm]{Remark}

\newtheorem{defn}[thm]{Definition}
\newtheorem{prop}[thm]{Proposition}
\newtheorem{cor}[thm]{Corollary}
\newtheorem{example}[thm]{Example}

\begin{document}

\title[Generalized solutions of PDEs and applications]{Generalized 
	solutions of variational \\ problems and applications}

\author{Vieri Benci}
\author{Lorenzo Luperi Baglini}
\author{Marco Squassina}

\address[M.\ Squassina]{Dipartimento di Matematica e Fisica \newline\indent
	Universit\`a Cattolica del Sacro Cuore \newline\indent
	Via dei Musei 41, I-25121 Brescia, Italy}
\email{marco.squassina@unicatt.it}

\address[V.\ Benci]{Dipartimento di Matematica\newline\indent 
Universit\`a degli Studi di Pisa\newline\indent
Via F. Buonarroti 1/c, 56127 Pisa, Italy\newline\indent
and Centro Linceo interdisciplinare Beniamino Segre\newline\indent 
Palazzo Corsini - Via della Lungara 10, 00165 Roma, Italy}
\email{benci@dma.unipi.it}

\address[L.\ Luperi Baglini]{Faculty of Mathematics\newline\indent
University of Vienna\newline\indent
Oskar-Morgenstern-Platz 1, 1090 Wien, Austria}

\email{lorenzo.luperi.baglini@univie.ac.at}

\subjclass[2010]{03H05, 26E35, 28E05, 46S20}
\keywords{Ultrafunctions, Non Archimedean Mathematics, Non Standard Analysis, Delta function.}

\thanks{M.\ Squassina is a member of the Gruppo Nazionale per l'Analisi Matematica, la Probabilit\`a e le loro Applicazioni (GNAMPA) and of Istituto Nazionale di Alta Matematica (INdAM). L.\ Luperi Baglini has been supported by grants M1876-N35, P26859-N25 and P 30821-N35 of the Austrian Science Fund FWF}

\begin{abstract}
Ultrafunctions are a particular class of generalized functions defined on a hyperreal field $\mathbb{R}^{*}\supset\mathbb{R}$ that allow to solve variational problems with no classical solutions. We recall the construction of ultrafunctions and we study the relationships between these generalized solutions and classical minimizing sequences. Finally, we study some examples to highlight the potential of this approach. 
\end{abstract}

\maketitle

\begin{center}
	\begin{minipage}{9.5cm}
		\small
		\tableofcontents
	\end{minipage}
\end{center}

\medskip

\section{Introduction\label{sec:Intro}}

It is nowadays very well known that, in many circumstances, the needs of a theory require the introduction of generalized functions. Among people working in partial differential equations, the theory of distributions of L.\ Schwartz is the most commonly used, but other notions of generalized functions have been introduced, e.g.\ by J.F.\ Colombeau \cite{col85} and
M.\ Sato \cite{sa59,sa60}. Many notions of generalized functions are based on non-Archimedean mathematics, namely mathematics handling infinite and/or infinitesimal quantities. Such an approach presents several positive features, the main probably being the possibility of treating distributions as non-Archimedean set-theoretical functions (under the limitations imposed by Schwartz' result). This allows to easily introduce nonlinear concepts, such as products, into distribution theory.  Moreover, a theory which includes infinitesimals and infinite quantities makes it possible to easily construct new models, allowing in this way to study several problems which are difficult even to formalize in classical mathematics. This has led to applications in various field, including several topics in analysis, geometry and mathematical physic (see e.g. \cite{Paolo,GKOS} for an overview in the case of Colombeau functions and their recent extension, called generalized smooth functions).

In this paper we deal with ultrafunctions, which are a kind of generalized functions that have been introduced recently in \cite{ultra} and developed in \cite{belu2012,belu2013,milano,algebra,beyond,BBG,burger,topology,gauss}.

Ultrafunctions are a particular case of non-Archimedean generalized functions that are based on the hyperreal
field $\mathbb{R}^{*},$ namely the numerical field on which nonstandard
analysis\footnote{We refer to Keisler \cite{keisler76} for a very clear exposition of nonstandard analysis.} is based. No prior knowledge of nonstandard analysis is requested to read this paper: we will introduce all the nonstandard notions that we need via a new notion of limit, called $\Lambda$-limit (see \cite{topology} for a complete introduction to this notion and its relationships with the usual nonstandard analysis). The main peculiarity of this
notion of limit is that it allows us to make a very limited use of formal logic, in constrast with most usual nonstandard analysis introductions. 

Apart from being framed in a non-Archimedean setting, ultrafunctions have other peculiar properties that will be introduced and used in the following: 

\begin{itemize}

\item every ultrafunction can be splitted uniquely (in a sense that will be precised in Section \ref{splitting}) as the sum of a classical function and a purely non-Archimedean part;
\item ultrafunctions extend distributions, in the sense that every distribution can be identified with an ultrafunction; in particular, this allows to perform nonlinear operations with distribution;
\item although being generalized functions, ultrafunctions shares many properties of $\mathcal{C}^{1}$ functions, like e.g.\ Gauss' divergence theorem.
\end{itemize}

Our goal is to introduce all the aforementioned properties of ultrafunctions so to be able to explain how they can be used to solve certain classical problems that do not have classical solutions; in particular, we will concentrate on singular problems arising in calculus of variations and in relevant applications (see, e.g., \cite{Paolo} and references therein for other approaches to these problems based on different notions of generalized functions).

The paper is organized as follows: in Section \ref{Lambda} we introduce the notion of $\Lambda$-limit, and we explain how to use it to construct all the non-Archimedean tools that are needed in the rest of the paper, in particular how to construct the non-Archimedean field extension $\mathbb{R}^{\ast}$ of $\mathbb{R}$ and what the notion of "hyperfinite" means. In Section \ref{sec:Ultrafunctions} we define ultrafunctions, and we explain how to extend derivatives and integrals to them. All the properties of ultrafunctions needed later on are introduced in this section: we show how to split an ultrafunction as the sum of a standard and a purely non-Archimedean part, how to extend Gauss' divergence theorem and how to identify distributions with certain ultrafunctions. In Section \ref{General theorems} we present the main results of the paper, namely we show that a very large class of classical problems admits a generalized ultrafunction solutions. We study the main properties of these generalized solutions, concentrating in particular on the relationships between ultrafunction solutions and classical minimizing sequences for variational problems. Finally, in Section \ref{applications} we present two examples of applications of our methods: the first is the study of a variational problem related with the sign-perturbation of potentials, the second is a singular variation problem related with sign-changing boundary conditions. 

The first part of this paper contains some overlap with other papers on ultrafunctions, but this fact is necessary to make it self-contained and to make the reader comfortable with it.

\subsection{Notations\label{sec:not}}

If $X$ is a set and $\Omega$\ is a subset of $\mathbb{R}^{N}$,
then 
\begin{itemize}
\item $\mathcal{P}\left(X\right)$ denotes the power set of $X$ and $\mathcal{P}_{{\rm fin}}\left(X\right)$
denotes the family of finite subsets of $X;$ 
\item $\mathfrak{F}\left(X,Y\right)$ denotes the set of all functions from
$X$ to $Y$ and $\mathfrak{F}\left(\Omega\right)=\mathfrak{F}\left(\Omega,\mathbb{R}\right).$ 
\item $\mathscr{\mathcal{C}}\left(\Omega\right)$ denotes the set of continuous
functions defined on $\Omega\subset\mathbb{R}^{N};$ 
\item $\mathcal{C}^{k}\left(\Omega\right)$ denotes the set of functions
defined on $\Omega\subset\mathbb{R}^{N}$ which have continuous derivatives
up to the order $k$ (sometimes we will use the notation $\mathscr{C}^{0}(\Omega)$
instead of $\mathscr{C}(\Omega)$ );
\item $H^{k,p}\left(\Omega\right)$ denotes the usual Sobolev space of functions
defined on $\Omega\subset\mathbb{R}^{N}$ 
\item if $W\left(\Omega\right)$ is any function space, then $W_{c}\left(\Omega\right)$
will denote de function space of functions in $W\left(\Omega\right)$
having compact support; 
\item $\mathcal{C}_{0}\left(\Omega\cup\Xi\right),\,\,\Xi\subseteq\partial\Omega,$
denotes the set of continuous functions in $\mathcal{C}\left(\Omega\cup\Xi\right)\ $
which vanish for $x\in\Xi$
\item $\mathcal{D}\left(\Omega\right)$ denotes the set of the infinitely
differentiable functions with compact support defined on $\Omega\subset\mathbb{R}^{N};\ \mathcal{D}^{\prime}\left(\Omega\right)$
denotes the topological dual of $\mathcal{D}\left(\Omega\right)$,
namely the set of distributions on $\Omega;$ 
\item if $A\subset X$ is a set, then $\chi_{A}$ denotes the characteristic function of $A;$ 
\item $\mathfrak{supp}(f)={\rm supp}^{\ast}(f)$ where ${\rm supp}$ is the usual notion
of support of a function or a distribution;
\item $\mathfrak{mon}(x)=\{y\in\left(\mathbb{R}^{N}\right)^{\ast}:x\sim y\}$
where $x\sim y$ means that $x-y$ is infinitesimal;
\item $\forall^{a.e.}$ $x\in X$ means "for almost every $x\in X";$ 
\item if $a,b\in\mathbb{R}^{*},$ then

\begin{itemize}
\item $\left[a,b\right]_{\mathbb{R}^{*}}=\{x\in\mathbb{R}^{*}:a\leq x\leq b\};$ 
\item $\left(a,b\right)_{\mathbb{R}^{*}}=\{x\in\mathbb{R}^{*}:a<x<b\};$ 
\end{itemize}
\item if $W$ is a generic function space, its topological dual will be
denoted by $W^{\prime}$ and the pairing by $\left\langle \cdot,\cdot\right\rangle _{W}$;
\item if $E$ is any set, then $|E|$ will denote its cardinality.
\end{itemize}

\section{$\Lambda$-theory\label{sec:lt}}\label{Lambda}

In this section we present the basic notions of Non Archimedean Mathematics (sometimes abbreviated as NAM)
and of Nonstandard Analysis (sometimes abbreviated as NSA) following a method inspired by \cite{BDN2003}
(see also \cite{ultra} and \cite{belu2012}).

\subsection{Non Archimedean Fields\label{naf}}

Here, we recall the basic definitions and facts regarding non-Archimedean
fields. In the following, ${\mathbb{K}}$ will denote a totally ordered infinite field.
We recall that such a field contains (a copy of) the rational numbers.
Its elements will be called numbers. 
\begin{defn}
Let $\mathbb{K}$ be an ordered field. Let $\xi\in\mathbb{K}$. We
say that: 
\end{defn}
\begin{itemize}
\item $\xi$ is infinitesimal if, for all positive $n\in\mathbb{N}$, $|\xi|<\frac{1}{n}$; 
\item $\xi$ is finite if there exists $n\in\mathbb{N}$ such as $|\xi|<n$; 
\item $\xi$ is infinite if, for all $n\in\mathbb{N}$, $|\xi|>n$ (equivalently,
if $\xi$ is not finite). 
\end{itemize}
\begin{defn}
An ordered field $\mathbb{K}$ is called Non-Archimedean if it contains
an infinitesimal $\xi\neq0$. 
\end{defn}
It is easily seen that infinitesimal numbers are actually finite, that the inverse
of an infinite number is a nonzero infinitesimal number, and that
the inverse of a nonzero infinitesimal number is infinite. 
\begin{defn}
A superreal field is an ordered field $\mathbb{K}$ that properly
extends $\mathbb{R}$. 
\end{defn}
It is easy to show, due to the completeness of $\mathbb{R}$, that
there are nonzero infinitesimal numbers and infinite numbers in any
superreal field. Infinitesimal numbers can be used to formalize a
new notion of closeness, according to the following

\begin{defn}
\label{def infinite closeness} We say that two numbers $\xi,\zeta\in{\mathbb{K}}$
are infinitely close if $\xi-\zeta$ is infinitesimal. In this case,
we write $\xi\sim\zeta$. 
\end{defn}
Clearly, the relation $\sim$ of infinite closeness is an equivalence
relation and we have the following 

\begin{thm}
If $\mathbb{K}$ is a totally ordered superreal field, every finite number $\xi\in\mathbb{K}$
is infinitely close to a unique real number $r\sim\xi$, called the
\textbf{standard part} of $\xi$. 
\end{thm}
Given a finite number $\xi$, we denote its standard part by $st(\xi)$,
and we put $st(\xi)=\pm\infty$ if $\xi\in\mathbb{K}$ is a positive
(negative) infinite number. In Definition \ref{def:St} we will see how the notion of standard part can be generalized to any Hausdorff topological space. 
\begin{defn}
Let $\mathbb{K}$ be a superreal field, and $\xi\in\mathbb{K}$ a
number. The {\em monad} of $\xi$ is the set of all numbers that are infinitely
close to it
\[
\mathfrak{m}\mathfrak{o}\mathfrak{n}(\xi):=\{\zeta\in\mathbb{K}:\xi\sim\zeta\}.
\]

\end{defn}

\subsection{The $\Lambda$-limit\label{subsec:Lambda-limit}}

Let $U$ be an infinite set of cardinality bigger that the continuum
and let $\mathfrak{L}=\mathscr{P}_{{\rm fin}}(U)$ be the family of finite
subsets of $U$.

Notice that $(\mathfrak{L},\subseteq)$ is a directed set. We add
to $\mathfrak{L}$ a \emph{point at infinity} $\Lambda\notin\mathfrak{L}$,
and we define the following family of neighbourhoods of $\Lambda:$
\[
\{\{\Lambda\}\cup Q\mid Q\in\mathcal{U}\},
\]
where $\mathcal{U}$ is a \emph{fine ultrafilter} on $\mathfrak{L}$,
namely a filter such that 
\begin{itemize}
\item for every $A,B\subseteq\mathfrak{L}$, if $A\cup B=\mathfrak{L}$
then $A\in\mathcal{U}$ or $B\in\mathcal{U}$; 
\item for every $\lambda\in\mathfrak{L}$ the set $Q(\lambda):=\{\mu\in\mathfrak{L}\mid\lambda\sqsubseteq\mu\}\in\mathcal{U}$. 
\end{itemize}
We will refer to the elements of $\mathcal{U}$ as qualified sets.
A function $\varphi:\,\mathfrak{L}\rightarrow E$ defined on a directed
set $E$ is called \emph{net} (with values in E). If $\varphi(\lambda)$
is a real net, we have that

\[
\lim_{\lambda\rightarrow\Lambda}\varphi(\lambda)=L
\]
if and only if\medskip{}
\[
\forall\varepsilon>0,\,\exists Q\text{\ensuremath{\in}}\mathcal{U}:
\,\,\,\,\forall\lambda\text{\ensuremath{\in}}Q,\,|\varphi(\lambda)-L|<\varepsilon.
\]

As usual, if a property $P(\lambda)$ is satisfied by any $\lambda$
in a neighbourhood of $\Lambda$ we will say that it is \textbf{\emph{eventually}}
satisfied. 

\begin{prop}
\label{nino}If the net $\varphi(\lambda)$ takes values in a compact
set $K$, then it is a converging net. 
\end{prop}
\begin{proof}
Suppose that the net $\varphi(\lambda)$ has a converging subnet to
$L\in\mathbb{R}$. We fix $\varepsilon>0$ arbitrarily and we have
to prove that $Q_{\varepsilon}\in\mathcal{U}$ where
\[
Q_{\varepsilon}=\left\{ \lambda\in\mathfrak{L}\ |\ \left\vert \varphi(\lambda)-L\right\vert <\varepsilon\right\} .
\]
We argue indirectly and we assume that 
\[
Q_{\varepsilon}\notin\mathcal{U}.
\]
 Then, by the definition of ultrafilter, $N=\mathfrak{L}\backslash Q_{\varepsilon}\in\mathcal{U}$
and hence
\[
\forall\lambda\in N,\ \left\vert \varphi(\lambda)-L\right\vert \geq\varepsilon.
\]
This contradict the fact that $\varphi_{\lambda}$ has a subnet which
converges to $L.$
\end{proof}

\begin{prop}
\label{prop:Ass}Assume that $\varphi:\,\mathfrak{L}\rightarrow E$
where $E$ is a first countable topological space; then if
\end{prop}
\[
\lim_{\lambda\rightarrow\Lambda}\varphi(\lambda)=x_{0},
\]
\emph{there exists a sequence $\left\{ \lambda_{n}\right\} $ in}
$\mathfrak{L}$\emph{ such that
\[
\lim_{n\rightarrow\infty}\varphi(\lambda_{n})=x_{0}
\]
We refer to the sequence $\varphi_{n}:=\varphi(\lambda_{n})$ as a
subnet of $\varphi(\lambda)$.}

\begin{proof}
Let $\{A_{n}\mid n\in\mathbb{N}\}$ be a countable basis of open neighbourhoods of $x_{0}$. For every $n\in\mathbb{N}$ the set 
\[I_{n}:=\{\lambda\in\mathfrak{L}\mid \varphi(\lambda)\in A_{n}\}\]
is qualified. Hence $J_{n}:=\bigcap_{j\leq n} I_{j}\neq\emptyset$. Let $\lambda_{n}\in J_{n}$. Then the sequence $\{\lambda_{n}\}_{n\in\mathbb{N}}$ has trivially the desired property: for every $n\in\mathbb{N}$, for every $m\geq n$ we have that $\varphi\left(\lambda_{m}\right)\in A_{n}$. 

\end{proof}

\begin{example}
Let $\varphi:\mathfrak{\,L}\rightarrow V$ be a net with values in a
bounded subset of a reflexive Banach space equipped with the weak topology;
then 
\[
v:=\lim_{\lambda\rightarrow\Lambda}\varphi(\lambda),
\]
is uniquely defined and there exists a sequence $n\mapsto\varphi(\lambda_{n})$
which converges to $v$. 
\end{example}

\begin{defn}
The set of the hyperreal numbers $\mathbb{R}^{*}\supset\mathbb{R}$
is a set equipped with a topology $\tau$ such that
\begin{itemize}
\item every net $\varphi:\,\mathfrak{L}\rightarrow\mathbb{R}$ has a unique
limit in $\mathbb{R}^{*}$ if $\mathfrak{L}$ and $\mathbb{R}^{*}$
are equipped with the $\Lambda$ and the $\tau$ topology respectively;
\item $\mathbb{R}^{*}$ is the closure of $\mathbb{R}$ with respect to
the topology $\tau$;
\item $\tau$ is the coarsest topology which satisfies the first property. 
\end{itemize}
\end{defn}

The existence of such $\mathbb{R}^{*}$ is a well known fact in NSA. The limit $\xi\in\mathbb{R}^{*}$ of a net $\varphi:\,\mathfrak{L}\rightarrow\mathbb{R}$ with respect to the $\tau$ topology, following \cite{ultra}, is called the $\Lambda$-limit of $\varphi$ and the following notation will be used:
\begin{equation}
\xi=\lim_{\lambda\uparrow\Lambda}\varphi(\lambda);\label{eq:lambdauno}
\end{equation}
namely, we shall use the up-arrow ``$\uparrow$'' to remind that
the target space is equipped with the topology $\tau$.
Given 
\[
\xi:=\lim_{\lambda\uparrow\Lambda}\varphi(\lambda),
\qquad
\eta:=\lim_{\lambda\uparrow\Lambda}\psi(\lambda),
\]
we set
\begin{equation}
\xi+\eta:=\lim_{\lambda\uparrow\Lambda}\left(\varphi(\lambda)+\psi(\lambda)\right),\label{eq:su}
\end{equation}
and 
\begin{equation}
\xi\cdot\eta:=\lim_{\lambda\uparrow\Lambda}\left(\varphi(\lambda)\cdot\psi(\lambda)\right).\label{eq:pr}
\end{equation}

Then the following well known theorem holds:
\begin{thm}
The definitions (\ref{eq:su}) and (\ref{eq:pr}) are well posed and
$\mathbb{R}^{*}$, equipped with these operations, is a non-Archimedean
field.
\end{thm}

\begin{rem}\rm
We observe that the field of hyperreal numbers is defined as a sort
of completion of real numbers. In fact $\mathbb{R}^{*}$ is isomorphic
to the ultrapower
\[
\mathbb{R}\mathfrak{^{L}/\mathfrak{I}}
\]
where
\[
\mathfrak{I}=\left\{ \varphi:\mathfrak{L}\rightarrow\mathbb{R}\,|\,\varphi(\lambda)=0 \ \text{eventually}\right\} 
\]
The isomorphism resembles the classical one between real numbers
and equivalence classes of Cauchy sequences: this method is surely
known to the reader for the construction of the real numbers starting
from the rationals. 
\end{rem}

\subsection{Natural extension of sets and functions}

To develop applications, we need to extend the notion of $\Lambda$-limit
to sets and functions (but also to differential and
integral operators). This will allow to enlarge the notions of variational problem and of variational
solution. 

$\Lambda$-limits of bounded nets of mathematical
objects in $V_{\infty}(\mathbb{R})$ can be defined by induction (a net $\varphi:\,\mathfrak{L}\rightarrow V_{\infty}(\mathbb{R})$
is called bounded if there exists $n\in\mathbb{N}$ such that $\forall\lambda\in\mathcal{\mathfrak{L}},\varphi(\lambda)\in V_{n}(\mathbb{R})$).
To do this, consider a net
\begin{equation}
\varphi:\mathcal{\mathfrak{L}}\rightarrow V_{n}(\mathbb{R}).\label{net}
\end{equation}

\begin{defn}
For $n=0,$ $\lim_{\lambda\uparrow\Lambda}\varphi(\lambda)$ is defined
by (\ref{eq:lambdauno}); so by induction we may assume that the limit
is defined for $n-1$ and we define it for the net (\ref{net}) as
follows:
\[
\lim_{\lambda\uparrow\mathcal{\textrm{\ensuremath{\Lambda}}}}\varphi(\lambda)=\left\{ \lim_{\lambda\uparrow\mathcal{\textrm{\ensuremath{\Lambda}}}}\psi(\lambda)\ |\ \psi:\mathfrak{\mathcal{\mathfrak{L}}}\rightarrow V_{n-1}(\mathbb{R}),\ \forall\lambda\in\mathcal{\mathfrak{L}},\ \psi(\lambda)\in\varphi(\lambda)\right\} .
\]
A mathematical entity (number, set, function or relation) which is
the $\Lambda$-limit of a net is called \textbf{internal}. 
\end{defn}

\begin{defn}
If $\forall\lambda\in\mathfrak{L,}$ $E_{\lambda}=E\in V_{\infty}(\mathbb{R}),$
we set $\lim_{\lambda\uparrow\Lambda}\ E_{\lambda}=E^{\ast},\ $namely
\[
E^{\ast}:=\left\{ \lim_{\lambda\uparrow\Lambda}\psi(\lambda)\ |\ \psi(\lambda)\in E\right\} .
\]
$E^{\ast}$ is called the \textbf{natural extension }of $E.$ 
\end{defn}
Notice that, while the $\Lambda$-limit\ of a constant sequence of
numbers gives this number itself, a constant sequence of sets gives
a larger set, namely $E^{\ast}$. In general, the inclusion $E\subseteq E^{\ast}$
is proper.

Given any set $E,$ we can associate to it two sets: its natural extension
$E^{\ast}$ and the set $E^{\sigma},$ where
\begin{equation}
E^{\sigma}:=\left\{ X^{\ast}\ |\ X\in E\right\}.
\label{sigmaS}
\end{equation}

Clearly $E^{\sigma}$ is a copy of $E;$ however it might be different
as set since, in general, $X^{\ast}\neq X.$ 

\begin{rem}\rm
If $\varphi:\mathfrak{\,L}\rightarrow X$ is a net with value in a
topological space we have the usual limit 
\[
\lim_{\lambda\rightarrow\Lambda}\varphi(\lambda),
\]
which, by Proposition \ref{nino}, always exists in the Alexandrov
compactification $X\cup\left\{ \infty\right\}$. Moreover we have
the $\Lambda$-limit, that always exists and it is en element of $X^{*}$.
In addition, the $\Lambda$-limit of a net is in $X^{\sigma}$ if
and only if $\varphi$ is eventually constant. If $X=\mathbb{R}$
and both limits exist, then 

\begin{equation}
\lim_{\lambda\rightarrow\Lambda}\varphi(\lambda)=st\left(\lim_{\lambda\uparrow\Lambda}\varphi(\lambda)\right).\label{eq:teresa}
\end{equation}
\end{rem}
The above equation suggests the following definition.
\begin{defn}
\label{def:St}If $X$ is topological space equipped with a Hausdorff
topology, and $\xi\in X^{*}$ we set
\[
St_{X}\left(\xi\right)=\lim_{\lambda\rightarrow\Lambda}\varphi(\lambda),
\]
if there is a net $\varphi:\mathfrak{\,L}\rightarrow X$ converging
in the topology of $X$ and such that
\[
\xi=\lim_{\lambda\uparrow\Lambda}\varphi(\lambda),
\]
and $St_{X}\left(\xi\right)=\infty$ otherwise.
\end{defn}

\noindent
By the above definition we have that 
\[
\lim_{\lambda\rightarrow\Lambda}\varphi(\lambda)=St_{X}\left(\lim_{\lambda\uparrow\Lambda}\varphi(\lambda)\right).
\]

\begin{defn}
Let 
\[
f_{\lambda}:\ E_{\lambda}\rightarrow\mathbb{R},\ \ \lambda\in\mathfrak{L},
\]
be a net of functions. We define a function
\[
f:\left(\lim_{\lambda\uparrow\Lambda}\ E_{\lambda}\right)\rightarrow\mathbb{R}^{*}
\]
as follows: for every $\xi\in\left(\lim_{\lambda\uparrow\Lambda}\ E_{\lambda}\right)$
we set
\[
f\left(\xi\right):=\lim_{\lambda\uparrow\Lambda}\ f_{\lambda}\left(\psi(\lambda)\right),
\]
where $\psi(\lambda)$ is a net of numbers such that 
\[
\psi(\lambda)\in E_{\lambda}\ \ \text{and}\ \ \lim_{\lambda\uparrow\Lambda}\psi(\lambda)=\xi.
\]
A function which is a $\Lambda$-\textit{limit\ is called }\textbf{\textit{internal}}\textit{.}
In particular if, $\forall\lambda\in\mathfrak{L,}$ 
\[
f_{\lambda}=f,\ \ \ \ f:\ E\rightarrow\mathbb{R},
\]
we set 
\[
f^{\ast}=\lim_{\lambda\uparrow\Lambda}\ f_{\lambda}.
\]
$f^{\ast}:E^{\ast}\rightarrow\mathbb{R}^{*}$ is called the \textbf{natural
extension }of $f.$ If we identify $f$ with its graph, then $f^{\ast}$
is the graph of its natural extension.
\end{defn}

\subsection{Hyperfinite sets and hyperfinite sums\label{HE}}
\begin{defn}
An internal set is called \textbf{hyperfinite} if it is the $\Lambda$-limit
of a net $\varphi:\mathfrak{L}\rightarrow\mathfrak{F}$ where $\mathfrak{F}$
is a family of finite sets. 
\end{defn}
For example, if $E\in V_{\infty}(\mathbb{R})$, the set
\[
\widetilde{E}=\lim_{\lambda\uparrow\Lambda}\left(\lambda\cap E\right)
\]
is hyperfinite. Notice that
\[
E^{\sigma}\subset\widetilde{E}\subset E^{*}
\]
so, we can say that every standard set is contained in a hyperfinite set. 

It is possible to add the elements of a hyperfinite set of numbers
(or vectors) as follows: let
\[
A:=\ \lim_{\lambda\uparrow\Lambda}A_{\lambda},
\]
be a hyperfinite set of numbers (or vectors); then the hyperfinite
sum of the elements of $A$ is defined in the following way: 
\[
\sum_{a\in A}a=\ \lim_{\lambda\uparrow\Lambda}\sum_{a\in A_{\lambda}}a.
\]
In particular, if $A_{\lambda}=\left\{ a_{1}(\lambda),...,a_{\beta(\lambda)}(\lambda)\right\} \ $with\ $\beta(\lambda)\in\mathbb{N},\ $then
setting 
\[
\beta=\ \lim_{\lambda\uparrow\Lambda}\ \beta(\lambda)\in\mathbb{N}^{\ast},
\]
we use the notation
\[
\sum_{j=1}^{\beta}a_{j}=\ \lim_{\lambda\uparrow\Lambda}\sum_{j=1}^{\beta(\lambda)}a_{j}(\lambda).
\]

\section{Ultrafunctions}\label{sec:Ultrafunctions}

\subsection{Definition of ultrafunctions\label{subsec:Cacc} }

We start by introducing the notion of hyperfinite grid:

\begin{defn}
A hyperfinite set $\Gamma$ such that $\mathbb{R}^{N}\subset\Gamma\subset\left(\mathbb{R}^{N}\right)^{*}$
is called hyperfinite grid.\end{defn}

From now on we assume that $\Gamma$ has been fixed once forever. Notice that, by definition, $\mathbb{R}^{N}\subseteq \Gamma$, and the following two simple (but useful) properties of $\Gamma$ can be easily proven via $\Lambda$-limits:

\begin{itemize}
\item for every $x\in\mathbb{R}^{N}$ there exists $y\in\Gamma\cap \mathfrak{mon}(x)$ so that $x\neq r$;
\item there exists a hyperreal number $\rho\sim 0,\rho>0$ such that $d(x,y)\geq \rho$ for every $x,y\in\Gamma, x\neq y$.
\end{itemize}

\begin{defn}
A space of grid functions is a family $\mathfrak{G}(\mathbb{R}^{N})$
of internal functions 
\[
u\,\,:\,\,\Gamma\rightarrow\mathbb{R}^{\ast}
\]
defined on a hyperfinite grid $\Gamma$. If $E\subset\mathbb{R}^{N}$,
then $\mathfrak{G}(E)$ will denote the restriction of the grid functions
to the set $E^{*}\cap\Gamma$.
\end{defn}
Let $E$ be any set in $\mathbb{R}^{N}$. To every internal function
$u\in\mathfrak{F}(E)^{*}$, it is possible to associate a grid function
by the ``restriction'' map
\begin{equation}
^{\circ}\,:\,\mathfrak{F}(E)^{*}\rightarrow\mathfrak{G}(E)\label{eq:tondino}
\end{equation}
defined as follows:
\[
\forall x\in E^{*}\cap\Gamma,\,\,u^{\circ}(x):=u^{*}(x);
\]
moreover, if $f\in\mathfrak{F}(E)$, for short, we use the notation
\begin{equation}
f^{\circ}(x):=\left(f^{*}\right)^{\circ}(x).\label{eq:lina}
\end{equation}
So every function $f\in\mathfrak{F}(E)$, can be uniquely extended
to a grid function $f^{\circ}\in\mathfrak{G}(E)$. 

In many problem we have to deal with functions defined almost everywhere
in $\Omega$, such as $1/|x|$. Thus it is useful to give a ``rule''
which allows to define a grid function for every $x\in\Gamma$:
\begin{defn}
\label{def:lina}If a function $f$ is defined on a set $E\subset\mathbb{R}^{N}$,
we put
\[
f^{\circ}(x)=\sum_{a\in\Gamma\cap E^{*}}f^{*}(a)\sigma_{a}(x),
\]
where $\forall a\in\Gamma$ the grid fuction $\sigma_{a}$ is defined
as follows: $\sigma_{a}(x):=\delta_{ax}.$ 
\end{defn}
If $E\subset\mathbb{R}^{N}$ is a measurable set, we define the ``density
function'' of $E$ as follows:
\begin{equation}
\theta_{E}(x)=st\left(\frac{m(B_{\eta}(x)\cap E^{*})}{m(B_{\eta}(x))}\right),\label{eq:estate}
\end{equation}
where $\eta$ is a fixed infinitesimal and $m$ is the Lebesgue measure.
Clearly $\theta_{E}(x)$ is a function whose value is 1 in $int(E)$
and 0 in $\mathbb{R}^{N}\setminus\overline{E}$; moreover, it is easy
to prove that $\theta_{E}(x)$ is a measurable function and we have
that
\[
\int\theta_{E}(x)dx=m(E)
\]
whenever $m(E)<\infty$; also, if $E$ is a bounded open set with
smooth boundary, we have that $\theta_{E}(x)=1/2$ for every $x\in\partial E.$ 

\vskip2pt
Now let $V(\mathbb{R}^{N})$ be a vector space such that $\mathscr{D}(\mathbb{R}^{N})\subset V(\mathbb{R}^{N})\subset\mathscr{L}^{1}(\mathbb{R}^{N})$.

\begin{defn}
A space of ultrafunctions $V^{\circ}(\mathbb{R}^{N})$ modelled
over the space $V(\mathbb{R}^{N})$ is a space of grid functions such
that there exists a vector space $V_{\Lambda}(\mathbb{R}^{N})\subset V^{*}(\mathbb{R}^{N})$
such that the map\footnote{$V^{*}(E)$ is a shorthand notation for $\left[V(E)\right]^{*}$ }
\[
^{\circ}:\ V_{\Lambda}(\mathbb{R}^{N})\rightarrow V^{*}(\mathbb{R}^{N})
\]
is a $\mathbb{R}^{*}$-linear isomophism. 
From now on, we assume that $V(\mathbb{R}^{N})$ satisfies the following
assumption: if $\Omega$ is a bounded open set such that $m_{N-1}(\partial\Omega)<\infty$
and $f\in C^{0}(\mathbb{R}^{N})$, then
\[
f\theta_{\Omega}\in V(\mathbb{R}^{N}).
\]
\end{defn}

Next we want to equip $V^{\circ}(\mathbb{R}^{N})$ with
the two main operations of calculus: the integral and the derivative.

\begin{defn}\label{integration}
\label{def:scorre}The \textbf{\emph{pointwise integral}} 
\[
\sqint\,\,:\ V^{\circ}(\mathbb{R}^{N})\rightarrow\mathbb{R}^{*}
\]
is a linear functional which satisfies the following properties:

\begin{enumerate}
\item $\forall u\in V_{\Lambda}(\mathbb{R}^{N})$ 
\begin{equation}
\sqint u^{\circ}(x)\,dx=\int u(x)dx;\label{eq:bimba}
\end{equation}
\item there exists a ultrafunction $d\,:\,\Gamma\rightarrow\mathbb{R}^{*}$
such that $\forall x\in\Gamma,\,d(x)>0$ and $\forall u\in V^{\circ}(\mathbb{R}^{N})$,
\[
\sqint u(x)\,dx=\sum_{a\in\Gamma}u(a)d(a).
\]
\end{enumerate}

\end{defn}

If $E\subset\mathbb{R}^{N}$ is any set, we use the obvious notation
\[
\sqint_{E}u(x)\,dx:=\sum_{a\in\Gamma\cap E^{*}}u(a)d(a)
\]

Few words to discuss the above definition. Point 2 says that the poinwise
integral is nothing else but a hyperfinite sum. Since $d(x)>0$ every
non-null positive ultrafunction has a strictly positive integral.
In particular if we denote by $\sigma_{a}(x)$ the ultrafunctions whose
value is $1$ for $x=a$ and $0$ otherwise, we have that
\[
\sqint\sigma_{a}(x)\,dx=d(a).
\]
The pointwise integral allows us to define the following scalar product:
\begin{equation}
\sqint u(x)v(x)dx=\sum_{a\in\Gamma}u(a)v(a)d(a).\label{eq:rina}
\end{equation}
From now on, the norm of an ultrafunction will be defined as
\[
\left\Vert u\right\Vert =\left(\sqint|u(x)|^{2}\ dx\right)^{\frac{1}{2}}.
\]
Now let us examine the point 1 of the above definition. If we take
$f\in C_{{\rm comp}}^{0}(\mathbb{R}^{N})$, we have that $f^{*}\in V_{\Lambda}(\mathbb{R}^{N})$
and hence
\[
\sqint f^{\circ}(x)\,dx=\int f(x)dx.
\]
Thus the pointwise integral is an extension of the Riemann Integral
defined on $C_{comp}^{0}(\mathbb{R}^{N})$. However, if we take a
bounded open set $\Omega$ such that $m(\partial\Omega)=0$, then
we have that
\[
\int_{\Omega}f(x)dx=\int_{\overline{\Omega}}f(x)dx;
\]
however, the pointwise integral cannot have this property; in fact
\[
\sqint_{\overline{\Omega}}f^{\circ}(x)dx-\sqint_{\Omega}f^{\circ}(x)dx=\sqint_{\partial\Omega}f^{\circ}(x)dx>0
\]
since $\partial\Omega\neq \emptyset$. In particular, if $\Omega$ is a bounded
open set with smooth boundary and $f\in C^{0}(\mathbb{R}^{N})$, then
\begin{align*}
\sqint_{\Omega}f^{\circ}(x)\,dx & =\sqint f^{\circ}(x)\chi_{\Omega}(x)\,dx\\
 & =\sqint f^{\circ}(x)\theta_{\Omega}^{\circ}(x)\,dx-\frac{1}{2}\sqint f^{\circ}(x)\chi_{\partial\Omega}^{\circ}(x)\,dx\\
 & =\int_{\Omega}f(x)\,dx-\frac{1}{2}\sqint_{\partial\Omega}f^{\circ}(x)\,dx,
\end{align*}
and similarly
\[
\sqint_{\overline{\Omega}}f^{\circ}(x)\,dx=\int_{\Omega}f(x)\,dx+\frac{1}{2}\sqint_{\partial\Omega}f^{\circ}(x)\,dx;
\]
of course, the term $\frac{1}{2}\sqint f^{\circ}(x)\chi_{\partial E}(x)\,dx$
is an infinitesimal number and it is relevant only in some particular
problems.

\begin{defn}
\label{def:deri}The \textbf{\emph{ultrafunction derivative}}
\[
D_{i}:\ V^{\circ} (\mathbb{R}^{N})\rightarrow V^{\circ} (\mathbb{R}^{N})
\]
is a linear operator which satisfies the following properties:

\begin{enumerate}
\item $\forall f\in C^{1}(\mathbb{R}^{N})$, and $\forall x\in(\mathbb{R}^{N})^{*}$,
$x$ finite, 
\begin{equation}
D_{i}f^{\circ}(x)=\partial_{i}f^{*}(x);\label{eq:ellera}
\end{equation}
\item $\forall u,v\in V^{\circ}(\mathbb{R}^{N})$, 
\[
\sqint D_{i}uv\,dx=-\sqint uD_{i}v\,dx;
\]
\item if $\Omega$ is a bounded open set with smooth boundary, then
$\forall v\in V^{\circ}$,
\[
\sqint D_{i}\theta_{\Omega}v\,dx=-\int_{\partial\Omega}^{*}v\,(\mathbf{e}_{i}\cdot\mathbf{n}_{E})\,dS
\]
where $\mathbf{n}_{E}$ is the unit outer normal, $dS$ is the $(n-1)$-dimensional measure
and $(\mathbf{e}_{1},....,\mathbf{e}_{N})$ is the canonical basis
of $\mathbb{R}^{N}$.
\item the support\footnote{If $u$ is an ultrafunction the support of $u$ is the set $\left\{ x\in\Gamma\,\,|\,\,u(x)\neq 0\right\}$.}
of $D_{i}\sigma_{a}$ is contained in $\mathfrak{mon}(a)\cap\Gamma$.
\end{enumerate}
\end{defn}

Let us comment the above definition. Point (1) implies that $\forall f\in C^{1}(\mathbb{R}^{N})$,
and $\forall x\in\mathbb{R}^{N}$, 
\begin{equation}
D_{i}f^{\circ}(x)=\partial_{i}f(x);\label{eq:ellera-2-1-1-1}
\end{equation}
namely the ultrafunction derivative\textbf{\emph{ }}concides with
the usual partial derivative whenever $f\in C^{1}(\mathbb{R}^{N})$.
The meaning of point (2) is clear; we remark that this point is very
important in comparing ultrafunctions with distributions. Point (3) says
that $D_{i}\theta_{\Omega}$ is an ultrafunction whose support is
contained in $\partial\Omega\cap\Gamma$; it can be considered as
a signed measure concentrated on $\partial\Omega$. Point (4) says that
the ultrafunction derivative, as well as the usual derivative or the
distributional derivative, is a local operator, namely if $u$ is
an ultrafunction whose support is contained in a compact set $K$
with $K\subset\Omega$, then the support of $D_{i}u$ is contained
in $\Omega^{*}$. Moreover, property (4) implies that the ultrafunction
derivative is well defined in $V^{\circ}(\Omega)$ for any
open set $\Omega$ by the following formula:
\[
D_{i}u(x)=\sum_{a\in\Gamma\cap\Omega^{*}}u(a)D_{i}\sigma_{a}(x).
\]

\begin{rem}\rm
If $u\in V^{\circ}(\Omega)$ and $\bar{u}$ is an ultrafunction
in $V^{\circ}(\overline{\Omega})$ such that $\forall x\in\Omega,\,\,u(x)=\bar{u}(x)$,
then, by point (3), $\forall x\in\Omega^{*}$ such that $\mathfrak{mon}(x)\subset\Omega^{*}$
we have that
\[
D_{i}u(x)=D_{i}\bar{u}(x);
\]
however, this property fails for some $x\sim\partial\Omega^{*}.$
In fact the support of $D_{i}\sigma_{a}$ is contained in $\mathfrak{mon}(a)\cap\Gamma$, but not in $\left\{ a\right\} $.
\end{rem}
\begin{thm}
There exists a ultrafuctions space $V^{\circ}(\mathbb{R}^{N})$
which admits a pointwise integral and a ultrafunction derivative as
in Def. \ref{def:scorre} and \ref{def:deri}.
\end{thm}
\begin{proof}
In \cite{BBG} there is a construction of a space $V_{\Lambda}(\mathbb{R}^{N})$
which satisfies the desired properties. The conclusion follows taking
\[
V^{\circ}(\mathbb{R}^{N})=\left\{ u^{\circ}:u\ensuremath{\in V_{\Lambda}}\right\}. \qedhere
\]
\end{proof}

\subsection{The splitting of an ultrafunction}\label{splitting}

In many applications, it is useful to split an ultrafunction $u$
in a part $w^{\circ}$ which is the canonical extension
of a standard function $w$ and a part $\psi$ which is not directly
related to any classical object.
If $u\in V^{\circ}(\Omega)$, we set
\[
S=\left\{ x\in\Omega\,|\ \text{$u(x)$ is infinite}\right\} 
\]
and
\[
w(x)=\begin{cases}
st(u(x)) & \text{if}\ x\in\Omega\setminus S;\\
0 & \text{if}\ x\in S.
\end{cases}
\]
We will refer to $S$ as to the singular set of the ultrafunction
$u$. 
\begin{defn}\label{splitting}
For every ultrafunction $u$ consider the splitting
\[
u=w^{\circ}+\psi
\]
where

\begin{itemize}
\item $w=\overline{w}_{|\Omega\setminus S}$ and $w^{\circ}$,
which is defined by Def. \ref{def:lina}, is called the \textbf{\emph{functional
part}} of $u$;
\item $\psi:=u-w^{\circ}$ is called the \textbf{\emph{singular
part}} of $u$.
\end{itemize}
\end{defn}
\medskip{}
Notice that $w^{\circ}$, the functional part of $u$, may
assume infinite values for some $x\in\Omega^{*}\setminus S^{*}$, but
they are determined by the values of $w$ which is a standard function
defined on $\Omega$.
\begin{example}
Take $\varepsilon\sim0,$ and 
\[
u(x)=\log(x^{2}+\varepsilon^{2}).
\]
In this case
\begin{itemize}
\item $w(x)=\begin{cases}
\log(x^{2})=2\,\log(|x|) & \text{if}\,\,x\neq0\\
0 & \text{if}\,\,x=0
\end{cases}$
\item $\psi(x)=\begin{cases}
\log(x^{2}+\varepsilon^{2})-\log(x^{2})=\log\left(1+\frac{\varepsilon^{2}}{x^{2}}\right) & \text{if}\ x\neq0\\
\log(\varepsilon^{2}) & \text{if}\ x=0
\end{cases}$
\item $S:=\left\{ 0\right\}.$
\end{itemize}
\end{example}
We conclude this section with the following trivial proposition which,
nevertheless, is very useful in applications: 
\begin{prop}
\label{prop:carola}Take a Banach space $W$ such that $\mathscr{D}(\Omega)\subset W\subset L^{1}(\Omega)$.
Assume that $\{u_{n}\}\subseteq V(\Omega)$ is
a sequence which converges weakly in $W$ and pointwise to a function
$w$; then, if we set
\[
u:=\left(\left(\lim_{\lambda\uparrow\Lambda}\ u_{|\lambda|}\right)^{^{\circ}}\right),
\]
we have that
\[
u=w^{\circ}+\psi,
\]
 where 
\[
\forall v\in W,\,\,\sqint\psi v\,dx\sim0;
\]
moreover, if
\[
\lim_{n\rightarrow\infty}\left\Vert u_{n}-w\right\Vert _{W}=0
\]
then $\left\Vert \psi\right\Vert _{W}\sim0.$ 
\end{prop}
\begin{proof}
As a consequence of the pointwise convergence of $\{u_{n}\}$ to $w$
we have that $\forall a\in\Gamma\,u(a)\sim w^{\circ}(a)$. In particular,
$\forall a\in\Gamma\,\psi(a)\sim0$. As $\Gamma$ is hyperfinite,
the set $\{|\psi(a)|\mid a\in\Gamma\}$ has a max $\eta\sim0$. Hence
for every $v\in W$ we have 
\[
\left|\sqint\psi v\,dx\right|\leq\sqint\left|\psi\right||v|\,dx\leq\eta\sqint|v|\,dx\sim\eta\int|v|dx\sim 0,
\]
as $\eta\sim0$ and $\int|v|dx\in\mathbb{R}$.
For the second statement let us notice that
\[
\left\Vert \psi\right\Vert _{W}=\left\Vert u-w^{\circ}\right\Vert _{W}=\left(\lim_{\lambda\uparrow\Lambda}\left\Vert u_{|\lambda|}-w\right\Vert _{W}\right)^{\circ}\sim0
\]
as $\lim_{n\rightarrow\infty}\left\Vert u_{n}-w\right\Vert _{W}=0$.
\end{proof}

An immediate consequence of Prop. \ref{prop:carola} is the following:
\begin{cor}
If $w\in L^{1}(\Omega)$ then
\[
\sqint w^{\circ}(x)dx\sim\intop w(x)dx.
\]
\end{cor}
\begin{proof}
Since $V(\Omega)$ is dense in $L^{1}(\Omega)$ there is a sequence
$u_{n}\in V(\Omega)$ which converges strongly to $w$ in $L^{1}(\Omega)$.
Now set 
\[
u:=\left(\lim_{\lambda\uparrow\Lambda}\ u_{|\lambda|}\right)^{\circ}.
\]
By Prop. \ref{prop:carola}, we have that 
\[
u=w^{\circ}+\psi
\]
with $\left\Vert \psi\right\Vert _{L^{1}}\sim0.$ Then 
\[
\sqint u(x)dx\sim\sqint w^{\circ}(x)dx.
\]
On the other hand, since $u\in V^{\circ}(\Omega),$ by Def. \ref{def:scorre}.(1),
\begin{align*}
\sqint u\text{\emph{(x)dx}} & =\int^{*}u^{\circ}(x)dx=\lim_{\lambda\uparrow\Lambda}\ \int u_{|\lambda|}dx\\
 & \sim\lim_{n\rightarrow\Lambda}\ \int_{\Omega}u_{n}dx=\intop w(x)dx.\qedhere
\end{align*}
\end{proof}

\subsection{The Gauss divergence theorem}

First of all, we fix the notation for the main differential operators: 
\begin{itemize}
\item $\nabla=(\partial_{1},...,\partial_{N})$ will denote the usual gradient
of standard functions;
\item $\nabla^{*}=(\partial_{1}^{*},...,\partial_{N}^{*})$ will denote
the natural extension of internal functions;
\item $D=(D_{1},...,D_{N})$ will denote the canonical extension of the
gradient in the sense of ultrafunctions.
\end{itemize}
Next let us consider the divergence:
\begin{itemize}
\item $\nabla\cdot\phi=\partial_{1}\phi_{1}+...+\partial_{N}\phi_{N}$ will
denote the usual divergence of standard vector fields $\phi\in\left[C^{1}(\mathbb{R}^{N})\right]^{N}$;
\item $\nabla^{*}\cdot\phi=\partial_{1}^{*}\phi_{1}+...+\partial_{N}^{*}\phi_{N}$
will denote the divergence of internal vector fields $\phi\in\left[C^{1}(\mathbb{R}^{N})^{*}\right]^{N}$;
\item $D\cdot\phi=D_{1}\phi_{1}+...+D_{N}\phi_{N}$ will denote the divergence
of vector valued ultrafuction $\phi\in\left[V^{\circ}(\mathbb{R}^{N})^{*}\right]^{N}$.
\end{itemize}
And finally, we can define the Laplace operator of an ultrafunction
$u\in V^{\circ}(\Omega)$ as the only ultrafunction $\bigtriangleup^{\circ}u\in V^{\circ}(\Omega)$
such that
\[
\forall v\in V_{0}^{\circ}(\overline{\Omega}),\,\,\,\sqint\bigtriangleup^{\circ}uvdx=-\sqint Du\cdot Dv\,dx,
\]
where
\[
V_{0}^{\circ}(\overline{\Omega}):=\left\{ v\in V^{\circ}(\overline{\Omega})\,\,|\,\,\forall x\in\partial\Omega\cap\Gamma,\,\,v(x)=0\right\} .
\]

\medskip{}

By definition \ref{def:deri}.(3), for any bounded open set $\Omega$
with smooth boundary,
\[
\sqint D_{i}\theta_{\Omega}v\,dx=-\int_{\partial\Omega}^{*}v\,(\mathbf{e}_{i}\cdot\mathbf{n}_{E})\,dS,
\]
and by definition \ref{def:deri}.(2)
\[
\sqint D_{i}\theta_{\Omega}v\,dx=-\sqint D_{i}v\,\theta_{\Omega}dx
\]
so that
\[
\sqint D_{i}v\,\theta_{\Omega}dx=\int_{\partial\Omega}^{*}v\,(\mathbf{e}_{i}\cdot\mathbf{n}_{\Omega})\,dS.
\]
Now, if we take a vector field $\phi=(v_{1},...,v_{N})\in\left[V^{\circ}(\mathbb{R}^{N})\right]^{N}$,
by the above identity, we get
\begin{equation}
\sqint D\cdot\phi\,\theta_{\Omega}\,dx=\int_{\partial\Omega}^{*}\phi\cdot\mathbf{n}_{\Omega}\,dS.\label{eq:camilla}
\end{equation}
Now, if $\phi\in C^{1},$ by definition \ref{def:deri}.(1), we get the divergence
Gauss theorem:
\[
\int_{\Omega}\nabla\cdot\phi\,dx=\int_{\partial\Omega}\phi\cdot\mathbf{n}_{E}\,dS.
\]
Then, (\ref{eq:camilla}) is a generalization of the Gauss' theorem
which makes sense for any bounded open set $\Omega$ with smooth boundary
and every vectorial ultrafunction $\phi$. Next, we want to generalize
Gauss' theorem to any subset of $A\subset\mathbb{R}^{N}$. 
It is well known that, for any bounded open set $\Omega$ with smooth
boundary, the distributional derivative $\nabla\theta_{\Omega}$ is
a vector valued Radon measure and we have that
\[
\left\langle |\nabla\theta_{\Omega}|,1\right\rangle =m_{N-1}(\partial\Omega)
\]
Then, the following definition is a natural generalization:
\begin{defn}
\label{def:misura}If $A$ is a measurable subset of $\mathbb{R}^{N}$,
we set
\[
m_{N-1}(\partial\Omega):=\sqint|D\theta_{A}^{\circ}|\,dx
\]
and $\forall v\in V^{\circ}(\mathbb{R}^{N}),$
\begin{equation}
\sqint_{\partial A}v(x)\,dS:=\sqint v(x)\,|D\theta_{A}^{\circ}|\,dx.\label{eq:nina}
\end{equation}
\end{defn}

\begin{rem}\rm
Notice that
\[
\sqint_{\partial A}v(x)\,dS\neq\sqint_{\partial A}v(x)\,dx.
\]
In fact the left hand term has been defined as follows:
\[
\sqint_{\partial A}v(x)\,dS=\sum_{x\in\Gamma}v(x)\,|D\theta_{A}^{\circ}(x)|\,d(x),
\]
while the right hand term is
\[
\sqint_{\partial A}v(x)\,dx=\sum_{x\in\Gamma\cap\partial A^{*}}v(x)\,d(x);
\]
in particular if $\partial A$ is smooth and $v(x)$ is bounded, $\sqint_{\partial A}v(x)\,dx$
is an infinitesimal number.
\end{rem}
\begin{thm}
\label{thm:41}If $A$ is an arbitrary measurable subset of $\mathbb{R}^{N}$,
we have that
\begin{equation}
\sqint D\cdot\phi\,\theta_{A}^{\circ}\,dx=\sqint_{\partial A}\phi\cdot\mathbf{n}_{A}^{\circ}(x)\,dS,\label{eq:gauus}
\end{equation}
where
\[
\mathbf{n}_{A}^{\circ}(x)=\begin{cases}
-\frac{D\theta_{A}^{\circ}(x)}{|D\theta_{A}^{\circ}(x)|} & \text{if}\,\,\,D\theta_{A}^{\circ}(x)\neq 0;\\
0 & \text{if}\,\,\,D\theta_{A}^{\circ}(x)=0.
\end{cases}
\]
\end{thm}
\begin{proof}
By Definition \ref{def:deri}.(3)

\[
\sqint D\cdot\phi\,\theta_{A}^{\circ}dx=-\sqint\phi\cdot D\theta_{A}^{\circ}dx,
\]
then, using the definition of $\mathbf{n}_{A}^{\circ}(x)$
and (\ref{eq:nina}), the above formula can be written as follows:
\[
\sqint D\cdot\phi\,\theta_{A}^{\circ}dx=\sqint\phi\cdot\mathbf{n}_{A}^{\circ}\:|D\theta_{A}^{\circ}|\,dx=\sqint_{\partial A}\phi\cdot\mathbf{n}_{A}^{\circ}\,dS.
\]
\end{proof}

\subsection{Ultrafunctions and distributions}

One of the most important properties of the ultrafunctions is that
they can be seen (in some sense that we will make precise in this
section) as generalizations of the distributions.
\begin{defn}
\label{DEfCorrespondenceDistrUltra}The space of \textbf{\emph{generalized
distribution}} on $\Omega$ is defined as follows: 
\[
\mathcal{\mathscr{D}}_{G}^{\prime}(\Omega)=V^{\circ}(\Omega)/N,
\]
where 
\[
N=\left\{ \tau\in V^{\circ}(\Omega)\ |\ \forall\varphi\in\mathscr{D}(\Omega),\ \int\tau\varphi\ dx\sim0\right\} .
\]
\end{defn}
The equivalence class of $u$ in $V^{\circ}(\Omega)$ will
be denoted by 
\[
\left[u\right]_{\mathscr{D}}.
\]

\begin{defn}
Let $\left[u\right]_{\mathfrak{\mathscr{D}}}$ be a generalized distribution.
We say that $\left[u\right]_{\mathscr{D}}$ is a bounded generalized
distribution if $\forall\varphi\in\mathfrak{\mathcal{\mathscr{D}}}(\Omega),\,\int u\varphi^{*}\ dx\ \ \text{is\ finite}$.
We will denote by $\mathscr{D}_{GB}^{\prime}(\Omega)$ the set of bounded generalized distributions.
\end{defn}

We have the following result.

\begin{thm}
\label{bello}There is a linear isomorphism 
\[
\Phi:\mathfrak{\mathcal{\mathscr{D}}}_{GB}^{\prime}(\Omega)\rightarrow\mathfrak{\mathscr{D}}^{\prime}(\Omega)
\]
such that, for every \textup{$\left[u\right]\in\mathfrak{\mathcal{\mathscr{D}}}_{GB}^{\prime}(\Omega)$
}\textup{\emph{and}}\textup{ }for every $\varphi\in\mathfrak{\mathscr{D}}(\Omega)$
\[
\left\langle \Phi\left(\left[u\right]_{\mathfrak{\mathcal{\mathscr{D}}}}\right),\varphi\right\rangle _{\mathfrak{\mathcal{\mathscr{D}}}(\Omega)}=st\left(\sqint u\,\varphi^{\ast}\ dx\right).
\]
\end{thm}
\begin{proof}
For the proof see e.g. \cite{algebra}.
\end{proof}
From now on we will identify the spaces $\mathfrak{\mathscr{D}}_{GB}^{\prime}(\Omega)$
and $\mathscr{D}^{\prime}(\Omega);$ so, we will identify $\left[u\right]_{\mathscr{D}}$
with $\Phi\left(\left[u\right]_{\mathscr{D}}\right)$ and we will
write $\left[u\right]_{\mathscr{D}}\in\mathfrak{\mathscr{D}}^{\prime}(\Omega)$
and 
\[
\left\langle \left[u\right]_{\mathscr{D}},\varphi\right\rangle _{\mathscr{D}(\Omega)}:=\langle\Phi[u]_{\mathscr{D}},\varphi\rangle=st\left(\sqint u\,\varphi^{\ast}\,dx\right).
\]

If $f\in C_{comp}^{0}(\Omega)$ and $f^{\ast}\in\left[u\right]_{\mathfrak{\mathscr{D}}}$,
then $\forall\varphi\in\mathfrak{\mathscr{D}}(\Omega)$,
\[
\left\langle \left[u\right]_{\mathfrak{\mathscr{D}}},\varphi\right\rangle _{\mathfrak{\mathscr{D}}(\Omega)}=st\left(\int^{\ast}u\,\varphi^{\ast}\ dx\right)=st\left(\int^{\ast}f^{\ast}\varphi^{\ast}dx\right)=\int f\varphi\,dx.
\]

\begin{rem}\rm
\label{rem:Schw}The set $V^{\circ}(\Omega)$ is an algebra
which extends the algebra of continuous functions $C^{0}(\mathbb{R}^{N})$.
If we identify a tempered distribution\footnote{We recall that, by a well known theorem of Schwartz, any tempered distribution
can be represented as $\partial^{m}f$ where $m$ is a multi-index
and $f$ is a continuous function.} $T=\partial^{m}f$ with the ultrafunction $D^{m}f^{\circ}$, we have
that the set of tempered distributions $\mathscr{S}'$ is contained
in $V^{\circ}(\Omega)$. However the Schwartz impossibility
theorem is not violated as $(V^{\circ}(\Omega),+,\,\cdot\,,\,D)$
is not a differential algebra since the Leibnitz rule does not hold
for some pairs of ultrafunctions. See also \cite{algebra}.
\end{rem}

\section{Properties of ultrafunction solutions}\label{General theorems}

The problems that we want to study with ultrafunctions have the following form: minimize a given functional $J$ on $V(\Omega)$ subjected to certain restrictions (e.g., some boundary constrictions, or a minimization on a proper vector subspace of $V(\Omega)$). This kind of problems can be studied in ultrafunctions theory by means of a modification of the Faedo-Galerkin method, based on standard approximations by finite dimensional spaces. The following is a (maybe even too) general formulation of this idea.

\begin{thm}\label{general existence}
	Let $W(\Omega)\neq\emptyset$ be a vector subspace of $V(\Omega)$. Let 
	\[\mathcal{F}=\{f:\,\,V(\Omega)\rightarrow\mathbb{R}\mid\forall E \text{ finite dimensional vector subspace of } W(\Omega)\ \exists u\in E \ f(u)=\min_{v\in E} f(v)\}.\]
	Then every $F\in\mathcal{F}^{\ast}$ has a minimizer in $W_{\Lambda}(\Omega)$.
\end{thm}

\begin{proof}
	Let $F=\lim_{\lambda\uparrow\Lambda}f_{\lambda}$, with $f_{\lambda}\in\mathcal{F}$
	for every $\lambda\in\mathfrak{L}$. By hypothesis, for every $\lambda\in\mathfrak{L}$
	there exists $u_{\lambda}\in W_{\lambda}:=Span(W\cap\lambda)$ that minimizes $f_{\lambda}$ on $W_{\lambda}$.
	Then $u=\lim_{\lambda\uparrow\Lambda}u_{\lambda}$	minimizes $F$ on $\lim_{\lambda\uparrow\Lambda}W_{\lambda}=W_{\Lambda}$ as, if $v=\lim_{\lambda\uparrow\Lambda} v_{\lambda}\in W_{\lambda}(\Omega)$, then for every $\lambda\in\mathfrak{L}$ we have that $f_{\lambda}\left( v_{\lambda} \right)\leq f_{\lambda} \left(u_{\lambda}\right)$, hence 
	
	\[ F(v)=\lim_{\lambda\uparrow\Lambda} f_{\lambda}\left( v_{\lambda} \right) \leq \lim_{\lambda\uparrow\Lambda} f_{\lambda}\left( u_{\lambda} \right) = F(u). \qedhere\]
\end{proof}

For applications, the following particular case of Theorem \ref{general existence} is particularly relevant:

\begin{cor} Let $f(\xi,u,x)$ be cohercive in $\xi$ on every finite dimensional subspace of $V(\Omega)$ and for every $x\in\Omega$. Let $F(u):=\sqint f(\nabla u, u, x)dx$. Then $F^{\circ}$ has a min on $V_{\Lambda}$. \end{cor}

\begin{proof} Just notice that $F\in\mathcal{F}$, in the notations of Theorem \ref{general existence}.\end{proof}

Theorem \ref{general existence} provides a general existence results. However, such a general result poses two questions: the first is how wild such generalized solutions can be; the second is if this method produces new generalized solutions for problems that already have classical ones.

The answer to these questions depends on the problem that is studied. However, regarding the second question we have the following result, which strengthens Theorem \ref{general existence}:

\begin{thm} \label{coherence} Let $F:V_{\Lambda}(\Omega)\rightarrow \mathbb{R}^{\ast}$, $F=\lim_{\lambda\uparrow\Lambda} F_{\lambda}$. For every $\lambda\in\mathfrak{L}$ let 

\[ M_{\lambda}:=\left\{u\in V_{\lambda}(\Omega)\mid F_{\lambda}(u)=\min_{v\in V_{\lambda}}F_{\lambda}(v)\right\}. \] 

Assume that $\lim_{\lambda\uparrow\Lambda} M_{\lambda}\neq\emptyset$. Then 
\[ M_{\Lambda}:=\left\{u\in V_{\Lambda}(\Omega)\mid  F(u)=\min_{v\in V_{\Lambda}}F(v)\right\}=\lim_{\lambda\uparrow\Lambda} M_{\lambda}\neq \emptyset. \] 
\end{thm}

\begin{proof} $M_{\Lambda}\subseteq\lim_{\lambda\uparrow\Lambda} M_{\lambda}$: Let $v=\lim_{\lambda\uparrow\Lambda} v_{\lambda}\in M_{\Lambda}$, and let $u=\lim_{\lambda\uparrow\Lambda} u_{\lambda}\in\lim_{\lambda\uparrow\Lambda} M_{\lambda}$. As $F(v)\leq F(u)$, there is a qualified set $Q$ such that for every $\lambda\in Q$ $F_{\lambda}\left(v_{\lambda}\right)\leq F\left(u_{\lambda}\right)$. But then $v_{\lambda}\in M_{\lambda}$ for every $\lambda\in Q$, hence $v=\lim_{\lambda\uparrow\Lambda}v_{\lambda}\in\lim_{\lambda\uparrow\Lambda} M_{\lambda}$.
	
	$M_{\Lambda}\supseteq\lim_{\lambda\uparrow\Lambda} M_{\lambda}$: Let $u=\lim_{\lambda\uparrow\Lambda} u_{\lambda}\in\lim_{\lambda\uparrow\Lambda} M_{\lambda}$. Let $v=\lim_{\lambda\uparrow\Lambda} v_{\lambda}\in V_{\Lambda}(\Omega)$. Let 
	\[Q=\{\lambda\in\mathfrak{L}\mid u_{\lambda}\in M_{\lambda}\}.\] Then $Q$ is qualified and, for every $\lambda\in Q$, $F_{\lambda}\left(u_{\lambda}\right)\leq F_{\lambda}\left(v_{\lambda}\right)$. Therefore $F(u)\leq F(v)$, and so $u\in M_{\Lambda}$.
\end{proof}

The following easy consequences of Theorem \ref{coherence} hold:
 
\begin{cor}In the same notations of Theorem \ref{coherence}, let us now assume that there exists $k\in\mathbb{N}$ such that $|M_{\lambda}|\leq k$ for every $\lambda\in\mathfrak{L}$. Then $|M_{\Lambda}|\leq k$.\end{cor}

\begin{proof} This holds as the hypothesis on $|M_{\lambda}|$ trivially entails that $|\lim_{\lambda\uparrow\Lambda} M_{\lambda}|\leq k$.\end{proof}

\begin{cor}\label{star are solutions} In the same notations of Theorem \ref{coherence}, let us now assume that $F=J^{\ast}$, where $J:V(\Omega)\rightarrow\mathbb{R}$. Let 
\[M:=\left\{v\in V(\Omega)\mid v=\min_{w\in V(\Omega)} J(w)\right\}.\]
Assume that $M\neq\emptyset$. Then the following facts are equivalent:
	
	\begin{enumerate}
		\item $u$ is a minimizer of $F:V_{\Lambda}(\Omega)\rightarrow\mathbb{R}^{\ast}$;
		\item $u\in M^{\ast}\cap V_{\Lambda}(\Omega)$.	
	\end{enumerate}
	
	In particular, if $u\in M$ then $u^{\ast}$ minimizes $F$. \end{cor}
	
\begin{proof} $(1)\Rightarrow (2)$ Let $u\in M$. Let $Q(u):=\{\lambda\in \mathfrak{L}\mid u\in \lambda\}$. Then for every $\lambda\in Q(u)$ $v\in M_{\lambda}\Leftrightarrow J(v)=J(u) \Rightarrow v\in M$, hence $M_{\lambda}\subseteq M$ for every $\lambda\in Q(u)$, which is qualified, and so $\lim_{\lambda\uparrow\Lambda} M_{\lambda}\subseteq M^{\ast}\cap V_{\Lambda}$, and we conclude by Theorem \ref{coherence}.

$(2)\Rightarrow (1)$ By definition,
\[u\in M^{\ast}\Leftrightarrow F(u)=\min_{v\in \left[V(\Omega)\right]^{\ast}} F(v),\]
hence if $u\in M^{\ast}\cap V_{\lambda}(\Omega)$ it trivially holds that $u$ minimizes $F$.
\end{proof}

\begin{cor} In the same hypotheses and notations of Corollary \ref{coherence}, let us assume that $M=\{u_{1},\dots,u_{n}\}$ is finite. Then $v$ minimizes $F$ in $V_{\Lambda}(\Omega)$ if and only if there exists $u\in M$ such that $u^{\ast}=v$.
\end{cor}

\begin{proof} Just remember that $S^{\circ}=\{s^{\ast}|s\in S\}$ for every finite set $S$, and that $M^{\sigma}=\{u^{\ast}\mid u\in M\}\subseteq V(\Omega)\subseteq V_{\Lambda}(\Omega)$. \end{proof}

In general, one might not have minima, but minimization sequences could still exist. In this case, we have the following result (in which for every $\rho\in\mathbb{R}^{\ast}$ we set $st_{\mathbb{R}}(\rho)=-\infty$ if and only if $\rho$ is a negative infinite number).  Notice that in the following result we are not assuming the continuity of $J$ with respect to any topology on $V(\Omega)$, in general.

\begin{thm}\label{minimizing sequences}Let $V(\Omega)$ be a Banach space, let $J:V(\Omega)\rightarrow \mathbb{R}$ and let $\inf_{u\in V(\Omega)}J(u)=m\in \mathbb{R}\cup\{-\infty\}$. The following facts hold:
\begin{enumerate}
  \item[1.] $J^{\ast}(v)\geq m$ for every $v\in V_{\Lambda}(\Omega)$.
  \item[2.] There exists $v\in V_{\Lambda}(\Omega)$ such that $st_{\mathbb{R}}(J^{\ast}(v))=m$.
	\item[3.] If $v\in V_{\Lambda}(\Omega)$ is a minimum of $J^{\ast}:V_{\Lambda}(\Omega)\rightarrow \mathbb{R}^{\ast}$ then $J^{\ast}(v)\geq st_{\mathbb{R}}\left(J^{\ast}(v)\right)=m$.
	\item[4.] Let $\{u_{n}\}_{n\in\mathbb{N}}$ be a minimizing sequence that converges to $u\in V(\Omega)$ in some topology $\tau$. Then there exists $v\in V_{\Lambda}(\Omega)$ such that $st_{\tau}(v)=u$ and $J^{\ast}(v)\geq st_{\mathbb{R}}\left(J^{\ast}(v)\right)=m$. Moreover, if $w^{\circ}+\psi$ is the canonical splitting of $v$, then: 
	\begin{itemize}
		\item if $\tau$ is the topology of pointwise convergence, then $w=v$ and $w(x)=u(x)$ for every $x\in\Omega$;
		\item if $\tau$ is the topology of pointwise convergence a.e., then $w=v$ and $w(x)=u(x)$ a.e. in $x\in\Omega$;
		\item if $\tau$ is the topology of weak convergence, then $w(x)=u(x)$ for every $x\in\Omega$ and $\langle \psi, \varphi^{\ast}\rangle^{\ast}\sim 0$ for every $\varphi$ in the dual of $V(\Omega)$;
		\item if $\tau$ is the topology associated with a norm $||\cdot||$ and, moreover, $\{u_{n}\}_{n}$ converges pointwise to $u$, then $w=u$ and $||\psi||^{\ast}\sim 0$.
	\end{itemize}
	\item[5.] If all minimizing sequences of $J$ converge to $u\in V(\Omega)$ in some topology $\tau$ and $v$ is a minimum of $J^{\ast}:V_{\Lambda}(\Omega)\rightarrow\mathbb{R}^{\ast}$ then $st_{\tau}(v)=u$ and $J^{\ast}(v)\geq st_{\mathbb{R}}(J^{\ast}(v))=m$.
\end{enumerate}
 
\end{thm}
\begin{proof} (1) Let $v=\lim_{\lambda\uparrow\Lambda} v_{\lambda}$. Since $m=\inf_{u\in V(\Omega)} J(u)$ we have that for every $\lambda\in\Lambda$ $J\left(v_{\lambda}\right)\geq m$, hence $J^{\ast}\left(v\right)\geq m$.

(2) By (1) it sufficies to show that $st_{\mathbb{R}}(J^{\ast}(v))=m$. Let $\{u_{n}\}_{n\in\mathbb{N}}$ be a minimizing sequence for $J$. For every $\lambda\in\mathfrak{L}$ let $v_{\lambda}:=u_{|\lambda|}$. Let $v:=\lim_{\lambda\uparrow\Lambda} v_{\lambda}$. We claim that $v$ is the desired ultrafunction.

To prove that $st_{\mathbb{R}}(J^{\ast}(v))=\lim_{n\rightarrow +\infty}J\left(u_{n}\right)=m$ we just have to observe that, by our definition of the net $\{v_{\lambda}\}_{\lambda}$, it follows that\footnote{A proof of this simple claim is given in Lemma 28 in \cite{burger}.} 

\[  \lim_{n\rightarrow +\infty}J\left(u_{n}\right)=st_{\mathbb{R}}\left( \lim_{\lambda\uparrow\Lambda} J\left(v_{\lambda}\right)\right), \]

and we conclude as $\lim_{\lambda\uparrow\Lambda} J\left(v_{\lambda}\right)=J^{\ast}(v)$ by definition.

(3) Let $v=\lim_{\lambda\uparrow\Lambda} v_{\lambda}$, and let $w\in V_{\Lambda}(\Omega)$ be such that $st_{\mathbb{R}}\left(J^{\ast}(w)\right)=m$. Then $m\leq J^{\ast}(v)$ by (1), whilst $st_{\mathbb{R}}\left(J^{\ast}(v)\right)\leq st_{\mathbb{R}}\left(J^{\ast}(w)\right)=m$. Hence $st\left(J^{\ast}(v)\right)=m$, as desired.

(4) Let $v$ be given as in point (2). Let us show that $st_{V(\Omega)}(v)=u$: let $A\in\tau$ be an open neighborhood of $u$. As $\{u_{n}\}_{n}$ converges to $u$, there exists $N>0$ such that for every $m>N$ $u_{n}\in A$. Let $\mu\in\mathfrak{L}$ be such that $|\mu|>N$. Then 

\[\forall\lambda\in Q_{\mu}:=\{\lambda\in\mathfrak{L}\mid \mu\subseteq\lambda\} \ v_{\lambda}\in A, \]

and as $Q_{\mu}$ is qualified, this entails that $v\in A^{\ast}$. Since this holds for every $A$ neighborhood of $u$, we deduce that $st_{\tau}(v)=u$, as desired.

Now let $u=w^{\circ}+\psi$ be the splitting of $u$. 

If $\tau$ is the pointwise convergence, $st_{\tau}(v)(x)=u(x)$ for every $x\in\Omega$, hence by Definition \ref{splitting} we have that the singular set of $u$ is empty and that $w(x)=u(x)$ for every $x\in\Omega$, as desired. A similar argument works in the case of the pointwise convergence a.e.

If $\tau$ is the weak convergence topology, then $st_{\tau}(v)=u$ means that $\langle v,\varphi^{\ast}\rangle^{\ast}\sim\langle u,\varphi\rangle$ for every $\varphi$ in the dual of $V(\Omega)$. Now let $S$ be the singular set of $u$. We claim that $S=\emptyset$. If not, let $x\in S$ and let $\varphi=\delta_{x}$. Then $\langle v,\varphi^{\ast}\rangle^{\ast}=v(x)$ is infinite, whilst $\langle u,\delta_{x}\rangle=u(x)$ is finite, which is absurd. Henceforth for every $x\in \Omega$ we have that $\psi(x)=0$. But 
\[\langle u,\varphi\rangle\sim\langle v,\varphi^{\ast}\rangle^{\ast}=\langle w^{\circ}+\psi,\varphi^{\ast}\rangle^{\ast}\\=\langle w^{\circ},\varphi^{\ast}\rangle^{\ast}+\langle \psi,\varphi^{\ast}\rangle^{\ast}\\=\langle w,\varphi\rangle+\langle \psi,\varphi^{\ast}\rangle^{\ast},\]

hence $st_{\tau}(\psi)=u-w$. As $\psi(x)=0$ for all $x\in\Omega$, this means that $u(x)=w(x)$ for every $x\in\Omega$. Then

\[ \langle u,\varphi\rangle+\langle \psi,\varphi^{\ast}\rangle^{\ast}=\langle w,\varphi\rangle+\langle \psi,\varphi^{\ast}\rangle^{\ast}=\langle v,\varphi\rangle\sim \langle u,\varphi\rangle, \]

and so $\langle \psi,\varphi^{\ast}\rangle^{\ast}\sim 0$.

Finally, if $\tau$ is the strong convergence with respect to a norm $||\cdot||$ and $\{u_{n}\}_{n}$ converges pointwise to $u$, then by what we proved above we have that $v(x)\sim u(x)\ \forall x\in \Omega$, hence $u(x)\sim w(x)$ for every $x\in\Omega$, which means $u=w$ as both $u,w\in V(\Omega)$. Then $||\psi||=||u-w^{\circ}||=||u-v^{\circ}||+||v^{\circ}-w^{\circ}||\sim 0$.

(5) Let $v=\lim_{\lambda\uparrow\Lambda} v_{\lambda}$. By point (2), the only claim to prove is that $st_{\tau}(v)=u$. We distinguish two cases:

Case 1: $J^{\ast}(v)\sim r\in\mathbb{R}$. As we noticed in point (2), it must be $r=m$. By contrast, let us assume that $st_{\tau}(v)\neq u$. In this case, there exists an open neighborhood $A$ of $u$ such that the set 

\[Q:=\left\{\lambda\in\mathfrak{L}\mid v_{\lambda}\notin A\right\} \]

is qualified. For every $n\in\mathbb{N}$, let 

\[Q_{n}:=\left\{\lambda\in\mathfrak{L}\mid |J(v_{\lambda})-r|<\frac{1}{n}\right\}\cap Q. \]

Every $Q_{n}$ is qualified, hence nonempty. For every $n\in\mathbb{N}$, let $\lambda_{n}\in Q_{n}$. Finally, set $u_{n}:=v_{\lambda_{n}}$. By construction, $\lim_{n\in\mathbb{N}}J\left(u_{n}\right)=m$. This means that $\{u_{n}\}_{n\in\mathbb{N}}$ is a minimizing sequence, hence it converges to $u$ in the topology $\tau$, and this is absurd as, for every $n\in\mathbb{N}$, by construction $u_{n}\notin A$. Henceforth $st_{\tau}(v)=u$. 

Case 2: $J^{\ast}(v)\sim -\infty$. As we noticed in the proof of point (2), in this case $m=-\infty$. Let us assume that $st_{V(\Omega)}(v)\neq u$. Then there exists an open neighborhood $A$ of $u$ such that the set 

\[Q:=\left\{\lambda\in\mathfrak{L}\mid v_{\lambda}\notin A\right\} \]

is qualified. For every $n\in\mathbb{N}$, let 

\[Q_{n}=\{\lambda\in\mathfrak{L}\mid J(v_{\lambda})<-n\}\cap Q\]

and let $\lambda_{n}\in Q_{n}$. Finally, let $u_{n}:=v_{\lambda_{n}}$. Then $J\left(u_{n}\right)<-n$ for every $n\in\mathbb{N}$, hence $\{u_{n}\}_{n\in\mathbb{N}}$ is a minimizing sequence, and so it must converge to $u$. However, by construction $u_{n}\notin A$ for every $n\in\mathbb{N}$, which is absurd.
\end{proof}

\begin{example} Let $\Omega=(0,1)$, let 

\begin{center} $V(\Omega)=\{u:\Omega\rightarrow\mathbb{R}\mid u$ is the restriction to $\Omega$ of a piecewise $\mathcal{C}^{1}([0,1])$ function$\}$  \end{center}

and let $J:V(\Omega)\rightarrow\mathbb{R}$ be the functional

\[ J(u):=\int_{\Omega} u^{2}(x)dx+\int_{\Omega}\left(\left(u^{\prime} \right)^{2}-1\right)^{2}dx. \]

It is easily seen that $\inf_{u\in V(\Omega)} J(u)=0$, and that the minimizing sequences of $J$ converge pointwise and strongly in the $L^{2}$ norm to $0$, but $J(0)=1$. 

Let $v\in V_{\Lambda}(\Omega)$ be the minimum of $J^{\ast}:V_{\Lambda}(\Omega)$. From points (4) and (5) of Theorem \ref{minimizing sequences} we deduce that $0< J^{\ast}(v)\sim 0$, that $st_{V(\Omega)}(v)=0$ and that the canonical decomposition of $v$ is $v=0^{\circ}+\psi$, with $\psi=0$ for every $x\in\Omega$ and $\int_{\Omega^{\ast}}^{\ast}\psi^{2} dx\sim 0$. Moreover, as $J^{\ast}(\psi)=0$, we also have that $\int^{\ast}_{\Omega^{\ast}}\left(\left(\psi^{\prime} \right)^{2}-1\right)^{2}dx\sim 0$. \end{example}

\section{Applications}\label{applications}

\subsection{Sign-perturbation of potentials}

\noindent The first problem that we would like to tackle by means of ultrafunctions
regards the sign-perturbation of potentials. 

Let us start by recalling some results recently proved by L.\ Brasco and M.\ Squassina in \cite{BS-17}
as a refinement and extension of some classical result by Brezis and Nirenberg \cite{BN-classic}.

\noindent Let $\Omega$ be a bounded domain of ${\mathbb{R}}^{N}$ with\footnote{In \cite{BS-17} the authors work more in general with a $p\in (1,N)$, and consider also a fractional version of Problem
\ref{eq: crit-probb}; however, in this paper we prefer to consider only the local case $p=2$.} $N>2$. Consider the minimization problem 
\begin{equation}\label{eq:classical Sp(a)}
\mathcal{S}(a):=\inf_{u\in\mathcal{D}_{0}^{1,2}(\Omega)}\left\{ \|\nabla u\|_{L^{2}(\Omega)}^{2}+\int_{\Omega}a\,|u|^{2}\,dx\,:\,\|u\|_{L^{2^{*}}(\Omega)}=1\right\} ,
\end{equation}
where $a\in L^{N/2}(\Omega)$ is given, $2^{*}=2N/(N-2)$, 
$$
\mathcal{D}_{0}^{1,2}(\Omega):=\big\{u\in L^{2^{\ast}}(\Omega)\mid \nabla u\in L^{2}(\Omega), u=0 \ \text{on} \ \partial\Omega\big\}.
$$
By Lagrange multipliers rule, minimizers of the previous problem (provided they
exist) are constant sign weak solutions of 
\begin{equation}
\begin{cases}
-\Delta u+a\,u=\mu\,|u|^{2^{\ast}-2}\,u, & \mbox{ in }\Omega,\\
u=0, & \mbox{ on }\partial\Omega,
\end{cases}\label{eq: crit-probb}
\end{equation}
with $\mu=\mathcal{S}(a)$, namely 
\[
\int_{\Omega}\nabla u\cdot\nabla\varphi\,dx+\int_{\Omega}a\,u\,\varphi\,dx=\mu\,\int_{\Omega}|u|^{2^{*}-2}\,u\,\varphi\,dx,
\]
for every $\varphi\in\mathcal{D}_{0}^{1,2}(\Omega)$. 

The main result in \cite{BS-17} is the following Theorem, where the
standard notations 
\[
a_{+}=\max\{a,\,0\},\qquad a_{-}=\max\{-a,0\},\qquad B_{R}(x_{0})=\{x\in\mathbb{R}^{N}\,:\,|x-x_{0}|<R\}
\]
are used: 
\begin{thm}[Brasco, Squassina]
	Let $\Omega\subset\mathbb{R}^{N}$ be an open bounded set. Then the
	following facts hold: 
	\begin{itemize}
		\item[1.] If $a\geq0$, then $\mathcal{S}(a)$ does not admit a solution. 
		\item[2.] Let $N>4$. Assume that there exist $\sigma>0,R>0$ and $x_{0}\in\Omega$
		such that 
		\[
		a_{-}\geq\sigma,\qquad\mbox{ a.\,e. on }B_{R}(x_{0})\subset\Omega.
		\]
		Then $\mathcal{S}(a)$ admits a solution. 
		\item[3.] Let $2<N\leq 4$.\ For any $x_{0}\in\Omega$, for any $R>0$ s.t. $B_{R}(x_{0})\subset\Omega$ there exists $\sigma=\sigma(R,N)>0$
		such that if 
		\[
		a_{-}\geq\sigma,\qquad\mbox{ a.\,e. on }B_{R}(x_{0}),
		\]
		then $\mathcal{S}(a)$ admits a solution. 
	\end{itemize}
\end{thm}

In \cite{ultra}, V.\,Benci studied in the ultrafunctions setting the following similar (simpler) problem: minimize
\begin{equation*}
\underset{u\in \mathfrak{M}_{p}}{\min }\ J(u),
\end{equation*}%
where%
\begin{equation*}
J(u)=\int_{\Omega }\left\vert \nabla u\right\vert ^{2}\ dx
\end{equation*}%
and 
\begin{equation*}
\mathfrak{M}_{p}=\left\{ u\in \mathcal{C}_{0}^{2}(\overline{\Omega }):\
\int_{\Omega }\left\vert u\right\vert ^{p}\ dx=1\right\}.
\end{equation*}%
Here $\Omega $ is a bounded set in $\mathbb{R}^{N}\ $with smooth boundary,\ $%
N\geq 3$ and $p>2$. In the ultrafunctions setting introduced in \cite{ultra} (and with the notations of \cite{ultra}), the problem takes
the following form:

\begin{equation}
\underset{u\in \widetilde{\mathfrak{M}}_{p}}{\min }\ J(u)  \label{*},
\end{equation}%
where%
\begin{equation*}
J(u)=\int_{\Omega }^{\ast }\left\vert \nabla u\right\vert ^{2}\ dx
\end{equation*}%
and 
\begin{equation*}
\widetilde{\mathfrak{M}}_{p}=\left\{ u\in V_{\mathcal{B}}^{2,0}(\overline{%
\Omega })\ |\ \int_{\Omega }^{\ast }\left\vert u\right\vert ^{p}\
dx=1\right\}
\end{equation*}%
with $V_{\mathcal{B}}^{2,0}(\overline{\Omega })=\mathcal{B}\left[ \mathcal{C%
}_{0}^{2}(\overline{\Omega })\right] .$

For every $p>2,$ problem (\ref{*}) has a ultrafunction solution $\tilde{u}_{p}$ and, by setting $\widetilde{m}_{p}=J(\tilde{u}_{p}),$ one can show that 

\begin{itemize}
\item (i) if $2<p<2^{\ast },$ then $\widetilde{m}_{p}=m_{p}\in \mathbb{R}%
^{+} $ and there is at least one standard minimizer $\tilde{u}_{p}$, namely $%
\tilde{u}_{p}\in \mathcal{C}_{0}^{2}(\overline{\Omega });$

\item (ii) if $p=2^{\ast }$ (and $\Omega \neq \mathbb{R}^{N}),$ then $%
\widetilde{m}_{2^{\ast }}=m_{2^{\ast }}+\varepsilon $ where $\varepsilon $
is a positive infinitesimal;

\item (iii) if $p>2^{\ast },$ then $\widetilde{m}_{p}=\varepsilon _{p}$
where $\varepsilon _{p}$ is a positive infinitesimal.
\end{itemize}

Our goal is to show that a similar result can be obtained for Problem \ref{eq: crit-probb}. 

In the present ultrafunctions setting, Problem \ref{eq: crit-probb} takes
the following form: find 
\begin{equation}\label{ultra crit-probb}
\widetilde{\mathcal{S}}(a):=\inf_{u\in V_{\Lambda}(\Omega)}\left\{ \int_{\left(\mathbb{R}^{N}\right)^{\ast}}^{\ast}|\nabla u|^{2}\,\,dx
+\int_{\left(\mathbb{R}^{N}\right)^{\ast}}^{\ast}a|u|^{2}\,\,dx\,:\,\|u\|_{\left[L^{2^{\ast}}(\mathbb{R}^{N})\right]^{\ast}}=1\right\} ,
\end{equation}
where $a\in\,^{\ast}\left[L^{N/2}(\Omega)\right]$ is given, 
and $V_{\Lambda}(\Omega)=\left[\mathcal{D}^{1,2}_{0}(\Omega)\right]_{\Lambda}$.
With the above notations, we can prove the following:

\begin{thm}
	Let $\Omega\subset\mathbb{R}^{N}$ be an open bounded set. Then the following facts hold:
	\begin{enumerate}
	  \item \label{enu: existence sign pert 0} For every $a\in\left[L^{N/2}(\Omega)\right]^{\ast}$ there exists	$u\in V_{\Lambda}(\Omega)$ that minimizes $\widetilde{\mathcal{S}}(a)$.
		\item \label{enu: existence sign pert 1} Let $a\in \left[L^{N/2}(\Omega)\right]$. If	$u\in\mathcal{C}^{1}\left( \Omega \right) \cap 
\mathcal{C}_{0}\left( \overline{\Omega }\right)$ is a minimizer of Problem \ref{eq:classical Sp(a)} then $u^{\ast}$ is a minimizer of $\widetilde{\mathcal{S}}(a^{\ast})$. 
		\item \label{enu: existence sign pert 2} If $a=0$, then $\widetilde{\mathcal{S}}(0)=S+\varepsilon$,
		where 
		$$
		S:=\inf_{u\in\mathcal{D}_{0}\left(\mathbb{R}^{N}\right)\setminus\{0\}}\frac{\|\nabla u\|^2_{L^2}}{||u||_{L^{2^{\ast}}}^{2}}
		$$ 
		and
		\[
		\varepsilon=\begin{cases}
		0, & \text{if}\ \Omega=\mathbb{R}^{N},\\
		\text{a strictly positive infinitesimal,} & \text{if}\ \Omega\neq\mathbb{R}^{N};
		\end{cases}
		\] 
	moreover, if $u$ is the minimizer in $V_{\Lambda}(\Omega)$, then the functional part $w$ of $u$ is $0$;
		\item \label{enu: existence sign pert 3} Let $a\geq 0$ have an isolated minimum $x_{m}$, and let $u\in V_{\Lambda}(\Omega)$ be the minimum of Problem \ref{ultra crit-probb}. If $u=w^{\circ}+\psi$ is the canonical splitting of $u$, then $w=0$ and $\psi$ concentrates in $x_{m}$, in the sense that for every $x\notin mon(x_{m}) \ \psi(x)\sim 0$. Moreover, $\langle \psi, \varphi^{\ast}\rangle^{\ast}\sim 0$ for every $\varphi$ in the dual of $V(\Omega)$.
		
	\end{enumerate}
	
\end{thm}

\begin{proof} (\ref{enu: existence sign pert 0}) This follows from Theorem \ref{coherence}, as the functional $\|\nabla u\|_{L^{2}(\Omega)}^{2}+\int_{\Omega}a\,|u|^{2}\,dx$ admits a minimum on every finite dimensional subspace of $V_{\Lambda}$.

(\ref{enu: existence sign pert 1}) This follows from Corollary \ref{star are solutions}.

(\ref{enu: existence sign pert 2}) In \cite[Lemma 3.1]{BS-17} it was proved that, if we consider	problem \ref{eq:classical Sp(a)}, we have that 
	\[
	\mathcal{S}(0)=S
	\]
	and $\mathcal{S}(0)$ is attained in $\mathcal{D}_{0}(\Omega)$ if and only if $\Omega =\mathbb{R}^{N}$. Therefore if $\Omega=\mathbb{R}^{N}$ the results follows from point (2). If $\Omega\neq\mathbb{R}^{N}$, the fact that $\widetilde{\mathcal{S}}(0)=S+\varepsilon$ follows from Theorem \ref{minimizing sequences}.(3). Moreover, all minimizing sequences $\{u_{n}\}_{n}$ converge weakly to $0$ in $H^{1}$, therefore they converges strongly in $L^{2}(\Omega)$ and so they converge pointwise a.e., hence by Theorem \ref{minimizing sequences}.(5) we deduce that, in the splitting $u=w^{\circ}+\psi$, we have that $w=0$, namely the ultrafunction solution coincides with its singular part. 
	
	(\ref{enu: existence sign pert 3}) We start by following the approach of \cite{BS-17}: we let $U$ be a minimizer of 
	$$
	\inf_{u\in\mathcal{D}_{0}^{1,2}(\Omega)}\frac{[u]_{\mathcal{D}^{1,2}}^2}{\|u\|_{L^{2^{\ast}}}^{2}}
	$$ 
	and, for every $\varepsilon>0$, let $U_{\varepsilon}(r):=\varepsilon^{\frac{2-N}{2}}U\left(\frac{r}{\varepsilon}\right)$. Let $\delta>0$ be such that $B_{\delta}(x_{m})\subseteq\Omega$, and let $u_{\delta,\varepsilon}$ be defined as follows
		\[u_{\delta,\varepsilon}=\begin{cases}
		U_{\varepsilon}(r) & \text{if}\ r\leq\delta,\\
		U_{\varepsilon}(\delta)\frac{U_{\varepsilon}(r)-U_{\varepsilon}(\delta\Theta)}{U_{\varepsilon}(\delta)-U_{\varepsilon}(\delta\Theta)} & \text{if}\ \delta< r\leq \delta\Theta;\\
		0 & \text{if}\ r>\delta\Theta;\\
		\end{cases}
		\] 
		where $\Theta$ is a constant given in Lemma 2.4 of \cite{BS-17}. Moreover, if $F(u):=\|\nabla u\|_{L^{2}(\Omega)}^{2}+\int_{\Omega}a\,|u|^{2}\,dx$, for $\delta_{1},\delta_{2}$ small enough we have that $F(u_{\delta_{1},\varepsilon})\leq F(u_{\delta_{2},\varepsilon})$. Then $\left(u_{\varepsilon,\varepsilon}\right)$ is a minimizing net (for $\varepsilon\rightarrow 0$), so we can use Theorem \ref{minimizing sequences}.(4): as $\left(u_{\varepsilon,\varepsilon}\right)$ converges pointwise to 0 we obtain that $w=0$, whilst the definition of the net ensure the concentration of $\psi$ in $x_{m}$. The last statement is again a direct consequence of Theorem \ref{minimizing sequences}.(4).
\end{proof}

Let us notice that the above Theorem shows a strong difference between
the ultrafunctions and the classical case: the existence of solutions
in $V_{\Lambda}(\Omega)$ is ensured independently of the sign of
$a$ whilst, as discussed in \cite[Section 4]{BS-17}, the conditions
on $a$ for the existence of solutions in the approach of L.\ Brasco and M.\ Squassina
are essentially optimal. Of course, ultrafunction solutions might
be very wild in general; their particular structure can be described in some cases, depending on $a$.

\subsection{The singular variational problem}

\subsubsection{Statement of the problem}

Let $W$ be a $C^{1}$-function defined in $\mathbb{R}\setminus\left\{ 0\right\} $
such that
\[
\underset{t\rightarrow0}{lim}\,W(t)=+\infty
\]
and
\[
\underset{t\rightarrow\pm\infty}{\overline{lim}}\,\,\frac{W(t)}{t^{2}}=0.
\]

We are interested in the singular problem (SP): \bigskip

\textbf{\emph{Naive formulation}} \textbf{\emph{of Problem SP}}: find
a continuous function
\[
u\,:\,\text{\ensuremath{\overline{\Omega}}}\rightarrow\mathbb{R},
\]
which satisfies the equation:
\begin{equation}
-\Delta u+W'(u)=0\,\,\,\text{in}\,\,\Omega\label{eq:a}
\end{equation}
with the following boundary condition:
\begin{equation}
u(x)=g(x)\,\,\,\text{for}\,\,x\in\partial\Omega,
\end{equation}
where $\Omega$ is an open set such that $\partial\Omega\neq\emptyset$
and $g\in L^{1}(\partial\Omega)$ is a function different from 0 for
every $x$ which change sign; e.g. $g(x)=\pm1$. Clearly, this problem
does not have any solution in $C^{1}$. This problem could be riformulated
as a kind of free boundary problem in the following way: \bigskip

\textbf{\emph{Classical formulation}} \textbf{\emph{of Problem SP}}:
find two open sets $\Omega_{1}$ and $\Omega_{2}$ and two functions
\[
u_{i}\,:\,\Omega_{i}\rightarrow\mathbb{R},\,\,i=1,2
\]
such that all the following conditions are fulfilled:

\[
\Omega=\Omega_{1}\cup\Omega_{2}\cup\Xi\,\,\text{where}\,\,\Xi=\overline{\Omega}_{1}\cap\overline{\Omega}_{2}\cap\Omega;
\]
\begin{equation}
-\Delta u_{i}+W'(u_{i})=0\,\,\,\text{in}\,\,\Omega_{i},\,\,i=1,2;\label{m}
\end{equation}
\[
u_{i}(x)=g(x)\,\,\,\text{for}\,\,x\in\partial\Omega\cap\partial\Omega_{i},\,\,i=1,2;
\]
\[
\underset{x\rightarrow\Xi}{\lim}u_{i}(x)=0;
\]
\begin{equation}
\Xi\,\,\text{is locally a minimal surface.}\label{eq:ms}
\end{equation}

Condition (\ref{eq:ms}) is natural, since formally equation (\ref{eq:a})
is the Euler-Lagrange equation relative to the energy
\begin{equation}
E(u)=\frac{1}{2}\int_{\Omega}\left(\left|\nabla u\right|^{2}+W(u)\right)dx\label{eq:ene-1}
\end{equation}
and the density of this energy diverges as $x\rightarrow\Xi$. In
general this problem is quite involved since the set $\Xi$ cannot
be a smooth surface and hence it is difficult to be characterized.
However this problem becomes relatively easy if formulated in the
framework of ultrafuctions.

Let us recall that we the Laplace operator of a ultrafuction $u\in V^{\circ}(\overline{\Omega})$ is defined
as the only ultrafunction $\Delta^{\circ}u\in V^{\circ}(\Omega)$
such that
\[
\forall v\in V_{0}^{\circ}(\overline{\Omega}),\,\,\,\sqint_{\Omega}\Delta^{\circ}uvdx=-\sqint_{\Omega}Du\cdot Dv\,dx
\]
where
\[
V_{0}^{\circ}(\overline{\Omega}):=\left\{ v\in V^{\circ}(\overline{\Omega})\,\,|\,\,\forall x\in\partial\Omega\cap\Gamma,\,\,v(x)=0\right\} .
\]
Notice that, we can assert that $\Delta^{\circ}u(x)=D\cdot D(x)$
only in $x\nsim\partial\Omega^{*}$.

\bigskip

\textbf{\emph{Ultrafunction formulation}} \textbf{\emph{of Problem
SP}}\footnote{If $u$ is a ultrafunction and $W,W'$, etc. are functions, for short,
we shall write $W(u),W'(u)$, etc instead of $W^{*}(u),\left(W'\right)^{*}(u),$
etc.}: find $u\in V^{\circ}(\overline{\Omega})$ such that
\begin{equation}
u(x)\neq0,\,\,\,\forall x\in\left(\overline{\Omega}\right)^{*}\cap\Gamma,\label{eq:c}
\end{equation}
\begin{equation}
-\Delta^{\circ}u+W'(u)=0,\,\,\,\text{for}\,\,x\in\Omega^{*}\cap\Gamma\label{eq:a-1}
\end{equation}
\begin{equation}
u(x)=g^{\circ}(x),\,\,\,\text{for}\,\,x\in\left(\partial\Omega\right)^{*}\cap\Gamma.\label{eq:b-1}
\end{equation}

As we will see in the next section, the existence of this problem
can be easily proven using variational methods.

\subsubsection{The existence result}

The easiest way to prove the existence of a ultrafunction solution
of problem SP is achieved exploiting the variational structure of
equation (\ref{eq:a-1}). Let us consider the extension
\begin{equation}
E^{\circ}(u)=\sqint_{\Omega}\left(\frac{1}{2}\left|Du\right|^{2}+W(u)\right)dx\label{eq:ene-1-1}
\end{equation}
 of the functional (\ref{eq:ene-1}) to the space 
\[
V_{g}^{\circ}(\overline{\Omega}):=\left\{ u\in V^{\circ}(\overline{\Omega})\,\,|\,\,\forall x\in\left(\partial\Omega\right)^{*}\cap\Gamma,\,\,u(x)=g(x)\right\}. 
\]

\begin{rem}
We remark that the integration is taken over $\Omega^{*}$, but $u$
is defined in $\overline{\Omega}^{*}$. This is important, in fact
for some $x\in\Omega^{*}$, $x\sim\partial\Omega^{*},$ the value
of $Du(x)$ depends on the value of $u$ in some point $y\in\partial\Omega^{*},\,y\sim x.$
This is a remarkable difference between the usual derivative and the
ultrafunction derivative.
\end{rem}

\begin{lemma}
Equation (\ref{eq:a-1}) is the Euler-Lagrange equation of the functional
(\ref{eq:ene-1-1}).
\end{lemma}

\begin{proof}
We use the expression of $\sqint$ as given in Definition \ref{def:deri}. As 
\[
E^{\circ}(u)=\sqint_{\Omega}\left(\frac{1}{2}\left|Du\right|^{2}+W(u)\right)dx
\]

let us compute separately the variations given by $\frac{1}{2}\left|Du\right|^{2}$ and $W(u)$. As 

\[\sqint_{\Omega^{\ast}}\frac{1}{2}\left|Du\right|^{2}dx=\sum_{a\in\Gamma\cap\Omega^{\ast}}\frac{1}{2}\left|Du(a)\right|^{2}da\]

is a quadratic form, for $v\in V_{0}^{\circ}(\overline{\Omega})$ we have that 

\[ \left(\frac{d}{du}\right)^{\ast}\left(\sqint_{\Omega^{\ast}} \frac{1}{2}\left|Du\right|^{2}dx\right)[v]=\sqint_{\Omega^{\ast}} DuDv dx =\\ \sqint_{\Omega^{\ast}} \left(-\Delta^{\circ}u\cdot v\right) dx.\]

The variation given by $W(u)$ for $v\in V_{0}^{\circ}(\overline{\Omega})$ is

\[ \left(\frac{d}{du}\right)^{\ast}\left(\sqint_{\Omega^{\ast}} \left(W(u)dx\right)\right)[v] = \left(\frac{d}{du}\right)^{\ast}\sum_{a\in\Gamma\cap\Omega^{\ast}} W(u(a))v(a)da = \sqint_{\Omega^{\ast}} W^{\prime}(u(x))v(x) dx. \]

Therefore the total variation of $E^{\circ}$ is 
\[
dE^{\circ}(u)\left[v\right]=\sqint_{\Omega^{\ast}}\left(-\Delta^{\circ} u+W'(u))\right)v\,dx,
\]

which proves our thesis.
\end{proof}

The existence of a ultrafunction solution of problem SP follows from the following Theorem:

\begin{lemma}\label{minimizer last problem}
The functional (\ref{eq:ene-1-1}) has a minimizer.
\end{lemma}
\begin{proof}
The functional $E^{\circ}(u)$ is coercive in the sense
that for any $c\in\mathbb{R}^{*}$
\[
E^{c}:=\left\{ u\in V_{g}(\overline{\Omega})\,\,|\,\,E^{\circ}(u)\leq c\right\} 
\]
is hypercompact (in the sense of NSA) since $V_{g}(\overline{\Omega})$
is a hypefinite dimensional affine manifold. Then, since $E^{\circ}$
is hypercontinuous (in the sense of NSA), the result follows.
\end{proof}

Regarding $\Xi$ being a minimal surface, we can prove the following:

\begin{prop}\label{jede}

Let $u$ be the ultrafunction minimizer of Problem \ref{eq:ene-1-1} as given by Lemma \ref{minimizer last problem}. Then the sets
\[
\Omega_{1}=\left\{ \,x\in\Omega\,\,|\,\,\forall y\in\mathfrak{mon}(x)\cap\Gamma,\,\,u(y)>0\right\}, 
\]
\[
\Omega_{2}=\left\{ \,x\in\Omega\,\,|\,\,\forall y\in\mathfrak{mon}(x)\cap\Gamma,\,\,u(y)<0\right\} 
\]
are open, hence 
\[
\Xi:=\left\{ \,x\in\overline{\Omega}\,\,|\,\,\exists y_{1},y_{2}\in\mathfrak{mon}(x)\cap\Gamma,\,u(y_{1})<0,\,u(y_{2})>0\right\} 
\]
is closed.

\end{prop}

\begin{proof} This follows from overspill\footnote{Overspill is a well known and very useful property in nonstandard analysis. The idea behind the version that we use here is the following: if a certain property $P(x)$ holds for every $x\sim 0$ then there must be a real number $r>0$ such that $P(x)$ holds for every $x<r$. For a proper formulation of overspill we refer to \cite{keisler76}.}: let us prove it for $\Omega_{1}$. Let $x\in\Omega_{1}$. By definition of $\Omega_{1}$, for every $\varepsilon\sim 0$ we have that $u(y)<0$ for every $y\in B_{\varepsilon}(x)\cap\Gamma$. Hence by overspill there exists a real number $r>0$ such that $u(y)<0$ for every $y\in B^{\ast}_{r}(x)\cap \Gamma$. As $B^{\sigma}_{r}(x)\subset B^{\ast}_{r}(x)\cap \Gamma$, we deduce that the open ball $B_{r}(x)\subseteq\Omega_{1}$.
 \end{proof}

Notice that Property (1) in Proposition \ref{jede} is a first step towards Property \ref{eq:ms} in the classical formulation of Problem SP. It is our conjecture, in fact, that $\Xi$ is a minimal surface, at least under some rather general hypothesis. We have not been able to prove this yet, however.

Let us conclude with a remark. When studying problems like Problem \ref{eq:ene-1-1} with ultrafunctions, one would like to be able to generalize certain properties of elliptic equations based on the maximum principle: for example, one would expect to have the following properties:

\begin{enumerate}
\item Let $\Omega$ be a bounded connected open set with smooth boundary and let $g$ be a bounded function. Then if
\[
u=w^{\circ}+\psi
\]
is the canonical splitting of $u$ as given in Definition \ref{splitting}, we have that $w\in L^{\infty}$ e $\psi(x)\sim 0$ for every $x\in\Omega$;
\item let $\Omega_{1},\Omega_{2}$ be the sets defined in Proposition \ref{jede}. Then in $\Omega_{1}\cup\Omega_{2}$ we have 
\[
-\Delta w+W'(w)=0,
\]
\[\Delta\psi(x)\sim 0;\]

\item if $a=inf(g)$ e $b=sup(g)$ and $W'(t)\geq0$ for all $t\in\mathbb{R}\setminus(a,b)$, we have that

\[ a\leq u(x)\leq b.\]
\end{enumerate}

However, in the spaces of ultrafunctions constructed in this paper the maximum principle does not hold directly: this is due to the fact that the kernel of the derivative is, in principle, larger than the space of constants. This problem could be avoided by modifying the space of ultrafunctions: as this leads to some technical difficulties, we prefer to postpone this study to a future paper.


\bigskip
\begin{thebibliography}{10}


\bibitem{ultra} 
V.\ Benci, \emph{Ultrafunctions and generalized
solutions,} Adv.\ Nonlinear Stud.\ 13, (2013), 461\textendash 486.

\bibitem{BBG}
V.\ Benci, L.\ Berselli, C.\ Grisanti, \emph{The Caccioppoli Ultrafunctions}, 
Adv. Nonlinear Anal., DOI: https://doi.org/10.1515/anona-2017-0225.


\bibitem{BDN2003} 
V.\ Benci, M.\ Di Nasso, \emph{Alpha-theory: an
elementary axiomatic for nonstandard analysis}, Expo.\ Math.\ 21, (2003),
355-386.

\bibitem{belu2012} 
V.\ Benci, L.\ Luperi Baglini, \emph{A model problem
for ultrafunctions}, in: Variational and Topological Methods: Theory,
Applications, Numerical Simulations, and Open Problems, Electron.\ J.\ Diff.\ Eqns., Conference 21 (2014), 11--21.

\bibitem{belu2013} V.\ Benci, L.\ Luperi Baglini, \emph{Basic Properties
of ultrafunctions,} in: Analysis and Topology in Nonlinear Dierential
Equations (D.G.\ Figuereido, J.M.\ do O, C.\ Tomei eds.), Progress
in Nonlinear Differential Equations and their Applications, 85 (2014),
61-86.

\bibitem{milano} 
V.\ Benci, L.\ Luperi Baglini, \emph{Ultrafunctions
and applications}, DCDS-S 7, 4, (2014), 593--616. 

\bibitem{algebra} 
V. Benci, L. Luperi Baglini, \emph{A non archimedean
algebra and the Schwartz impossibility theorem,}, Monatsh. Math. (2014),
503--520.

\bibitem{beyond} 
V.\ Benci, L.\ Luperi Baglini, \emph{Generalized
functions beyond distributions}, AJOM 4, (2014).

\bibitem{gauss} 
V.\ Benci, L.\ Luperi Baglini, \emph{A generalization
of Gauss' divergence theorem}, in: Recent Advances in Partial Dierential
Equations and Applications, Proceedings of the International Conference
on Recent Advances in PDEs and Applications (V.\ D.\ Radulescu, A.\ Sequeira,
V.\ A.\ Solonnikov eds.), Contemporary Mathematics (2016), 69--84.

\bibitem{topology} V.\ Benci, L.\  Luperi Baglini, \emph{A topological approach to non-Archimedean
mathematics,} in: ``Geometric Properties for Parabolic and Elliptic
PDE''' (F.\ Gazzola, K.\ Ishige, C.\ Nitsch, P.\ Salani eds.), Springer
Proceedings in Mathematics \& Statistics, Vol.\ 176 (2016), 17\textendash 40,
doi 10.1007/978-3-319-41538-3\_2.


\bibitem{burger} V.\ Benci, L.\ Luperi Baglini, \emph{Generalized solutions in PDE's
and the Burgers' equation,} J.\ Differential Equations 263, Issue 10, 15 November 2017, 6916--6952.


\bibitem{boccardo} 
L.\ Boccardo, G.\ Croce, \emph{Elliptic partial
differential equations}, De Gruyter, (2013).

\bibitem{BS-17} 
L.\ Brasco, M.\ Squassina, 
\emph{Optimal solvability for a nonlinear problem at critical growth},
J.\ Differential Equations 264 (2018), 2242--2269.

\bibitem{BN-classic} 
H.\ Brezis, L.\ Nirenberg, 
\emph{Positive solutions of nonlinear elliptic equations involving critical Sobolev exponents},
Comm. Pure Appl. Math. {\bf 36} (1983), 437--477.

\bibitem{col85}
J.F.\ Colombeau, \emph{Elementary introduction to
new generalized functions}. North-Holland Mathematics Studies, 113.
Notes on Pure Mathematics, 103. North-Holland Publishing Co., Amsterdam,
1985


\bibitem{GKOS}Grosser, M., Kunzinger, M., Oberguggenberger, M., Steinbauer,
R., Geometric theory of generalized functions, Kluwer, Dordrecht (2001).


\bibitem{keisler76} 
H.J.\ Keisler, \emph{Foundations of
Infinitesimal Calculus}, Prindle, Weber \& Schmidt, Boston, (1976).

\bibitem{Paolo}
A.\ Lecke, L.\ Luperi Baglini, P.\ Giordano, \emph{The classical theory
of calculus of variations for generalized smooth functions}, Adv. Nonlinear Anal.,
doi.org/10.1515/anona-2017-0150.

\bibitem{nelson}
E.\ Nelson, \emph{Internal Set Theory: A new approach
to nonstandard analysis}, Bull.\ Amer.\ Math.\ Soc., 83 (1977), 1165\textendash 1198.

\bibitem{rob} 
A.\ Robinson, \emph{Non-standard Analysis,}Proceedings
of the Royal Academy of Sciences, Amsterdam (Series A) 64, (1961),
432--440.

\bibitem{sa59}
M.\ Sato, \emph{, Theory of hyperfunctions.} II.\ J.\
Fac.\ Sci.\ Univ. Tokyo Sect. I 8 (1959) 139--193.

\bibitem{sa60}
M.\ Sato,  \emph{Theory of hyperfunctions}. II.\ J.\ Fac.\
Sci.\ Univ.\ Tokyo Sect. I 8 (1960) 387\textendash 437.

\bibitem{squa}
M.\ Squassina, \emph{Exstence, multiplicity, perturbation
and concentration results for a class of quasi linear elliptic problems,}
Electronic Journal of Differential Equations, Monograph 07, 2006,
(213 pages). ISSN: 1072-6691. 
\end{thebibliography}
\end{document}